\documentclass[final,a4paper]{amsart}
\usepackage{amssymb,amsmath,amsthm}
\usepackage{color,graphicx,float}
\usepackage{amsfonts}
\usepackage[foot]{amsaddr}

\usepackage{amsmath,amssymb,amsfonts}
\usepackage[foot]{amsaddr}
\usepackage{verbatim}
\usepackage{graphicx,inputenc}
\usepackage{subfigure}
\usepackage[notref,notcite]{showkeys}

\newtheorem{theorem}{Theorem}[section]
\newtheorem{lemma}{Lemma}[section]
\newtheorem{definition}[theorem]{Definition}

\newtheorem{assumption}[theorem]{Assumption}
\theoremstyle{remark}
\newtheorem{remark}{Remark}[section]
\newtheorem{corollary}{Corollary}[section]
\numberwithin{equation}{section}

\def\Dt{\partial_t^\alpha}
\def\Dmu{D^{(\mu)}}
\def\Imu{I^{(\mu)}}
\def\N+{n\in\mathbb{N}^{+}}

\def\FT{\mathcal{F}}

\def\re{\operatorname{Re}}

\begin{document}

\title[recovering the distributed fractional derivative]{Fractional diffusion:
recovering the distributed fractional derivative from overposed data}

\author{ William Rundell \ and Zhidong Zhang}
\address{Department of Mathematics. Texas A\&M University, USA}
\email{rundell@math.tamu.edu\quad zhidong@math.tamu.edu}
\date{}
\maketitle

\begin{abstract}
There has been considerable recent study in ``subdiffusion'' models
that replace the standard parabolic equation model by a one with
a fractional derivative in the time variable.
There are many ways to look at this newer approach and one such is to
realize that the order of the fractional derivative is related to
the time scales of the underlying diffusion process.
This raises the question of what order $\alpha$ of derivative should be taken
and if a single value actually suffices.
This has led to models that combine a finite number of these derivatives
each with a different fractional exponent $\alpha_k$ and different
weighting value $c_k$ to better model a greater possible
range of time scales.
Ultimately, one wants to look at a situation that combines
derivatives in a continuous way -- the so-called distributional model
with parameter $\mu(\alpha)$.

However all of this begs the question of how one determines this
``order'' of differentiation.
Recovering a single fractional value has been an active part of the process
from the beginning of fractional diffusion modeling and if this is the
only unknown then the markers left by the fractional order derivative
are relatively straightforward to determine.
In the case of a finite combination of derivatives this
becomes much more complex due to the more limited analytic tools available
for such equations, but recent progress in this direction has been made,
\cite{LiYamamoto:2015,LiLiuYamamoto:2015}.
This paper considers the full distributional model where the order
is viewed as a function $\mu(\alpha)$ on the interval $(0,1]$.
We show existence, uniqueness and regularity for an initial-boundary value
problem including an important representation theorem in the case of a single
spatial variable.
This is then used in the inverse problem of recovering the
distributional coefficient $\mu(\alpha)$ from a time trace of the solution
and a uniqueness result is proven.\\

\noindent\text{Keywords}: distributed-order fractional diffusion, 
uniqueness, inverse problem.
\\\\
\text{AMS subject classifications}: 35R30, 26A33, 60C22, 34A08.
\end{abstract}
\maketitle


\section{Introduction}\label{sec:intro}

Classical Brownian motion as formulated in
Einstein's 1905 paper \cite{Einstein:1905b}
can be viewed as a random walk
in which the dynamics are governed by an uncorrelated,
Markovian, Gaussian stochastic process.
The key assumption is that a change in the direction
of motion of a particle is random and that the mean-squared displacement
over many changes is proportional to time
$\langle x^2\rangle = C t$.
This easily leads to the derivation of the underlying differential equation
being the heat equation.

In fact we can generalize this situation to the case of a
continuous time random walk ({\sc ctrw}) where
the length of a given jump, as well as the waiting time elapsing
between two successive jumps follow a given probability density function.
In one spatial dimension, the picture is as follows:
a walker moves along the $x$-axis, starting at a position $x_0$ at time $t_0=0$.
At time $t_1$, the walker jumps to $x_1$, then at time $t_2$ jumps to $x_2$,
and so on.
We assume that the temporal and spatial increments
$\,\Delta t_n = t_n - t_{n-1}$, $\,\Delta x_n = x_n-x_{n-1}$
are independent, identically distributed random variables,
following probability density functions $\psi(t)$ and $\lambda (x)$,
respectively, which is known as the waiting time distribution and jump
length distribution, respectively.
Namely, the probability of $\Delta t_n$ lying in any interval $[a,b]\subset
(0,\infty)$ is
$ P(a<\Delta t_n<b) = \int_a^b \psi(t)\,dt$
and the probability of $\Delta x_n$ lying in any interval
$[a,b]\subset \mathbb{R}$ is $P(a<\Delta x_n <b) = \int_a^b \lambda(x)\,dx$.
For given $\psi$ and $\lambda$, the position $x$ of the walker
can be regarded as a step function of $t$.

It it easily shown using the Central Limit Theorem
that provided the first moment, or characteristic waiting time $T$, defined by
$T = \mu_1(\psi)=\int_0^\infty t\psi(t)\,dt$ and the second moment,
or jump length variance $\Sigma$,
$\mu_2(\lambda)=\int_{-\infty}^\infty x^2\lambda(t)\,dt$ are finite,
then the long-time limit again corresponds to Brownian motion,

On the other hand,
when the random walk involves correlations,
non-Gaussian statistics or a non-Markovian process
(for example, due to ``memory'' effects)
the diffusion equation will fail to describe the macroscopic limit.
For example, if we retain the assumption that $\Sigma$ is finite but
relax the condition on a finite characteristic waiting time so that
for large $t$ $\psi(t)  A/t^{1+\alpha}$ as $t\to\infty$ where
$0<\alpha\leq 1$ , then we get very different results.
Such probability density functions are often referred to as
a ``heavy-tailed.''
If in fact we take
\begin{equation}\label{eqn:frac_dist}
\psi(t) = \frac{A_\alpha}{B_\alpha+t^{1+\alpha}}
\end{equation}
then again it can be shown, \cite{MontrollWeiss:1965,KlafterSokolov:2011},
that the effect is to modify the Einstein formulation
$\langle x^2\rangle = C t$ to $\langle x^2\rangle = C t^\alpha$.

This above leads to a {\it subdiffusive\/} process and, importantly provides
a tractable model where the partial differential equation is replaced by
one with a fractional derivative in time of order $\alpha$.
Such objects have been a steady source of investigation over the last
almost 200 years beginning in the 1820s with the work of Abel
and continuing first by Liouville then by Riemann.

The fractional derivative operator can take several forms, the most usual
being either the Riemann-Liouville $^R\!D_0^\alpha$ based on Abel's original
singular integral operator,
or the Djrbashian-Caputo  $^C\!D^\alpha_0$ version,
\cite{Djrbashian:1989}, which reverses the order of the 
Riemann-Liouville formulation
\begin{equation}\label{eqn:frac_ders}
\begin{aligned}
^R\!D^\alpha_0 u &=\frac{1}{\Gamma(n-\alpha)}\frac{d^n}{dx^n}\int_0^x(x-t)^{\alpha
	+1-n}u(t)\,dt,\\
^C\!D^\alpha_0u &= \frac{1}{\Gamma(n-\alpha)}\int_0^x(x-t)^{\alpha+1-n}u^{(n)}(t)\,dt.\\
\end{aligned}
\end{equation}
The D{\v{z}}rba{\v{s}}jan-Caputo derivative tends to be more favored by
practitioners since it allows the
specification of initial conditions in the usual way.
Nonetheless, the Riemann-Liouville derivative enjoys certain analytic
advantages, including being defined for a wider class of functions
and possessing a semigroup property.

Thus the fractional-anomalous diffusion model gives rise to the fractional
differential equation
\begin{equation} \label{eqn:basic_one_term}
\partial_t^\alpha u - \mathcal{L} u = f(x,t),\qquad
x\in\Omega, t\in (0,T)
\end{equation}
where $\mathcal{L}$ is a uniformly elliptic differential operator on 
an open domain $\Omega\subset\mathbb{R}^d$ and $\partial_t^\alpha$ is 
one of the above fractional derivatives.
The governing function for the fractional derivative becomes
the Mittag-Leffler function $E_{\alpha,\beta}(z)$
which generalizes the exponential function that forms the key component
for the fundamental solution in the classical case when $\alpha=\beta=1$.
\begin{equation}\label{eqn:mlf}
E_{\alpha,\beta}(z) = \sum_{k=0}^\infty \frac{z^k}{\Gamma(\alpha k + \beta)}
\end{equation}
For the typical examples described here we have $0<\alpha\leq 1$ and
$\beta$ a positive real number although further generalization is certainly
possible.
See, for example, \cite{GorenfloKilbasMainardiRogosin:2014}.

During the past two decades, differential equations involving
fractional-order derivatives have received increasing attention in
applied disciplines.
Such models are known to capture more faithfully the dynamics of anomalous
diffusion processes in amorphous materials,
e.g., viscoelastic materials, porous media, diffusion on domains with
fractal geometry and option pricing models.
These models also describe certain diffusion processes more accurately
than Gaussian-based Brownian motion and have particular relevance
to materials exhibiting memory effects.
As a consequence, we can obtain fundamentally different physics.
There has been significant progress on both
mathematical methods and numerical algorithm design and, more recently,
attention has been paid to inverse problems.
This has shed considerable light on the new physics appearing,
\cite{JinRundell:2015,SokolovKlafterBlumen:2002}

Of course, such a specific form for $\psi(t)$ as given by \ref{eqn:frac_dist}
is rather restrictive as it assumes a quite specific scaling factor between 
space and time distributions and there is no reason to expect nature
is so kind to only require a single value for $\alpha$.

One approach around this is to take a finite sum of such
terms each corresponding to a different value of $\alpha$.
This leads to a model where the time derivative is replaced by
a finite sum of fractional derivatives of orders $\alpha_j$ and
by analogy leads to the law  $\langle x^2\rangle = g(t,\alpha)$
where $g$ is a finite sum of fractional powers.
This formulation replaces the single value fractional derivative by a finite sum
$\sum_1^m q_j\partial_t^{\alpha_j} u$ where a linear combination of
$m$ fractional powers has been taken.
Physically this represents a fractional diffusion model
that assumes diffusion takes place in a medium in
which there is no single scaling exponent; for example,
a medium in which there are memory effects over multiple time scales.

This seemingly simple device leads to considerable complications.
For one, we have to use the so-called multi-index Mittag-Leffler function
$E_{\alpha_1,\,\ldots\,\alpha_m,\beta_1,\,\ldots\,\beta_m}(z)$ in place
of the two parameter $E_{\alpha,\beta}(z)$ and this adds complexity not only
notationally but in proving required regularity results for the basic
forwards problem of knowing $\Omega$, $\mathcal{L}$, $f$, $u_0$
and recovering $u(x,t)$, see \cite{LiYamamoto:2015,LiLiuYamamoto:2015} and the references within.

It is also possible to generalize beyond the finite sum by taking the so-called
distributed fractional derivative, 
\begin{equation}\label{eqn:distributional_der-def}
\partial_t^{(\mu)} u(t) = \int_0^1 \mu(\alpha) \partial_t^\alpha u(t) \,d\alpha.
\end{equation}
Thus the finite sum derivative can be obtained by taking
$\mu(\alpha) = \sum_{j=1}^m q_j\delta(\alpha-\alpha_j)$.
See \cite{Naber:2003,Kochubei:2008,MMPG:2008,Luchko:2009b,li2016analyticity},
for several studies incorporating this extension.
This in turn allows a more general function probability
density distribution function $\psi$ in \ref{eqn:frac_dist} and hence
a more general value for $g(t,\alpha)$.

The purpose of this paper is to analyze this distributed model extension to
equation~\eqref{eqn:basic_one_term} and the paper is organized as follows.
First, we demonstrate existence, uniqueness and regularity results for the
solution of the distributed fractional derivative model on a cylindrical region
in space-time $\Omega\times[0,T]$ where $\Omega$ is a bounded, open set in
$\mathbb{R}^d$. Second, in the case of one spatial variable, $d=1$,
we set up representation theorems for the solution analogous to that for
the heat equation itself, \cite{Cannon:1984}, and extended to the
case of a single fractional derivative in \cite{RundellXuZuo:2013}.

Section~\ref{sec:eur} looks at the assumptions to be made on the various terms
in \eqref{eqn:distributional_der-def} and utilizes these to show existence,
uniqueness and regularity results for the direct problem;
namely, to be given $\Omega$, $\mathcal{L}$, $f$, $u_0$ and the function
$\mu=\mu(\alpha)$, then to solve \eqref{eqn:distributional_der-def}
for $u(x,t)$.

Section~\ref{sect:representation} will derive several representation
theorems for this solution and these will be used in the final section
to formulate and prove a uniqueness result for the associated inverse problem
to be discussed below.
\medskip

However, there is the obvious question for all of these models:
what is the value of $\alpha$?
Needless to say there has been much work done on this; experiments have
been set up to collect additional information that allows a best fit
for $\alpha$ in a given setting.
One of the earliest works here is from 1975, \cite{Scher_Montroll:1975} and
in part was based on the Montroll-Weiss random walk model
\cite{MontrollWeiss:1965}.
See also \cite{HatanoHatano:1998}.
Mathematically the recovery in models with
a single value for $\alpha$ turns out to relatively
straightforward provided we are able to choose the type of data being measured.
This would be chosen to allow us to rely on the known asymptotic behavior
of the Mittag-Leffler function for both small and large arguments.
An exception here is when we also have to determine $\alpha$ as well
as an unknown coefficient in which case the combination problem can
be decidedly much more complex. 
See, for example,  \cite{cheng2009uniqueness, Li2013Simultaneous, RundellXuZuo:2013}.
Amongst the first papers in this direction with a rigorous existence and
uniqueness analysis is \cite{HatanoNakagawaWangYamamoto:2013}.

The multi-term case, although similar in concept,
is quite nontrivial but has been shown in
\cite{LiYamamoto:2015,LiLiuYamamoto:2015}.
In these papers the authors were able to prove an important uniqueness
theorem: if given the additional data consisting of the value of the normal
derivative $\frac{\partial u}{\partial\nu}$ at a fixed point
$x_0\in\partial\Omega$ for all $t$ then the sequence pair
$\{q_j,\alpha_j\}_{j=1}^m$ can be uniquely recovered.

The main result of the current paper in this direction is 
in Section~\ref{inverse_problem} where we show that
the uniqueness results of
\cite{LiYamamoto:2015,LiLiuYamamoto:2015}
can be extended to recover a suitably defined exponent function $\mu(\alpha)$.

\section{Preliminary material}\label{sec:eur}

Let $\Omega$ be an open bounded domain in $\mathbb{R}^d$
with a smooth ($C^2$ will be more than sufficient) boundary $\partial\Omega$
and let $T>0$ be a fixed constant.

$\mathcal{L}$ is a strongly elliptic, self-adjoint operator with
smooth coefficients defined on $\Omega$,
\begin{equation*}
\mathcal{L} u = \sum_{i,j=1}^d a_{ij}(x)\frac{\partial^2 u}{\partial x_i\partial x_j}
+ c(x) u
\end{equation*}
where $a_{ij}(x)\in C^1(\overline{\Omega})$, $c(x)\in C(\overline{\Omega})$, $a_{ij}(x)=a_{ji}(x)$ and
$\sum_{i,j=1}^d a_{ij}\xi_i\xi_j \geq \delta \sum_{i=1}^d \xi^2$ for some 
$\delta>0$, all $x\in\overline\Omega$ and all
$\xi=(\xi_1,\,\ldots\,\xi_d)\in \mathbb{R}^d$.

To avoid unnecessary complications for the main theme we will make the 
assumption of homogeneous Dirichlet boundary conditions on $\partial\Omega$
so that the natural domain for $\mathcal{L}$ is $H^2(\Omega)\cap H_0^1(\Omega)$.
Then $-\mathcal{L}$ has a complete, orthonormal system
of eigenfunctions $\{\psi_n\}_1^\infty$ in $L^2(\Omega)$
with $\psi_n\in H^2(\Omega)\cap H_0^1(\Omega)$
and with corresponding eigenvalues $\{\lambda_n\}$ such that
$0<\lambda_1\leq \lambda_2\leq\dots \leq\lambda_n \to \infty$ as $n\to\infty$.

The nonhomogeneous term will be taken to satisfy $f(x,t)\in C(0,T; H^2(\Omega))$.
This can be weakened to assume only $L^p$ regularity in time, but as shown in
\cite{LiLiuYamamoto:2015} this requires more delicate analysis.
The initial value $u_0(x)\in H^2(\Omega)$.
We will use $\langle\cdot,\cdot\rangle$ to denote the inner product in
$L^2(\Omega)$.

Throughout this paper we will, by following \cite{Kochubei:2008}, make the assumptions on the distributed
derivative parameter $\mu.$

\noindent
\begin{assumption}\label{mu_assumption}
$$\mu\in C^1[0,1],\ \mu(\alpha)\ge 0,\ \mu(1)\ne 0.$$ 
\end{assumption}
\begin{remark}
From these conditions it follows that
there exists a constant $C_\mu>0$ and an interval
$(\beta_0,\beta)\subset (0,1)$ such that
$\mu(\alpha)\ge C_\mu$ on $(\beta_0,\beta)$.
This will be needed in our proof of the representation theorem in 
Section~\ref{sect:representation}.
\end{remark}
We will use the Djrbashian-Caputo version for $\Dmu$:
$\Dmu u=\int_0^1 \mu(\alpha){}\Dt u {\rm d}\alpha$
with
$\Dt u=\frac{1}{\Gamma(1-\alpha)}\int_0^t(t-\tau)^{-\alpha}
\frac{d}{d\tau}u(x,\tau){\rm d}\tau$
and so 
\begin{equation}\label{eqn:dist_frac_der}
 D^{(\mu)}u=\int_0^t \left[\int_0^1 \frac{\mu(\alpha)}{\Gamma(1-\alpha)}
   (t-\tau)^{-\alpha} {\rm d}\alpha\right]\frac{d}{d\tau}u(x,\tau){\rm d}\tau
   := \int_0^t \eta(t-\tau)\frac{d}{d\tau}u(x,\tau){\rm d}\tau,
\end{equation}
where
\begin{equation}\label{eqn:dist_frac_eta}
\eta(s)=\int_0^1 \frac{\mu(\alpha)}{\Gamma(1-\alpha)}
  s^{-\alpha} {\rm d}\alpha.
\end{equation}

Thus our distributed differential equation (DDE) model in this paper will be
\begin{equation}\label{eqn:model_pde}
\begin{aligned}
\Dmu u(x,t) - \mathcal{L} u(x,t) &= f(x,t), &&\quad x\in\Omega,\quad t\in(0,T);\\
u(x,t) &= 0, &&\quad x\in\partial\Omega,\quad t\in(0,T);\\
u(x,0) &= u_0(x), &&\quad x\in\Omega.\\
\end{aligned}
\end{equation}

\subsection{A Distributional ODE}

Our first task is to analyze the ordinary distributed fractional order equation
\begin{equation}\label{eqn:dODE}
 \Dmu v(t)=-\lambda v(t),\ v(0)=1,\ t\in (0,T)
\end{equation}
and to show there exists a unique solution.
We will need some preliminary analysis to determine the integral operator
that serves as the inverse for $D^{(\mu)}$ in analogy with
the Riemann-Liouville derivative being inverted by the Abel operator.
If we now take the Laplace transform of $\eta$ in \eqref{eqn:dist_frac_eta}
then we have 
\begin{equation}\label{eqn:Phi}
 (\L\eta) (z)=\frac{\Phi(z)}{z},\quad \text{where }\ 
\Phi(z)=\int_0^1 \mu(\alpha)z^\alpha {\rm d}\alpha.
\end{equation}

\par The next lemma introduces an operator $\Imu$ to analyze 
the distributed ODE \eqref{eqn:dODE}.
\begin{lemma}\label{lem:kappa}
 Define the operator $\Imu$ as 
 $$
\Imu \phi(t)=\int_0^t \kappa(t-s)\phi(s){\rm d}s,\quad \text{where }\ 
\kappa (t)=\frac{1}{2\pi i} \int_{\gamma-i\infty}^{\gamma+i\infty}
 \frac{e^{zt}}{\Phi(z)}{\rm d}z.
$$
Then the following conclusions hold: 
\begin{itemize}
 \item [(1)] $\Dmu\Imu \phi(t)=\phi(t),\ \Imu\Dmu \phi(t)=\phi(t)-\phi(0)$ 
 for $\phi\in C^1(0,T);$
 \item [(2)] $\kappa(t)\in C^\infty (0,\infty)$ and
\begin{equation}\label{eqn:kappa_inequality}
\kappa(t)=|\kappa(t)|\le C \ln {\frac{1}{t}}\ \ \text{for sufficiently small}\ t>0.
\end{equation}
 \end{itemize}
\end{lemma}

\begin{proof}
 This is \cite[Proposition~3.2]{Kochubei:2008}.
We remark that the result in this paper include further estimates
on $\kappa$ that require additional regularity on $\mu$.
However, for the bound \eqref{eqn:kappa_inequality} only $C^1$ regularity on
$\mu$ is needed.
\end{proof}

\begin{remark}
In \cite[Proposition~3.2]{Kochubei:2008}, if the condition either
$\mu(0)\ne 0$ or $\mu(\alpha)\sim a\alpha^v,\ a>0,\ v>0$ is added,
then $\kappa$ is completely monotone.
This property is not explicitly used in this paper, however as we remark after the
uniqueness result, this condition on $\kappa$ could be a useful basis for
a reconstruction algorithm.
\end{remark}

\par With $\Imu$, we have the following results.
\begin{lemma}\label{lem:existence_uniqueness_u_n}
For each $\lambda>0$ there exists a unique $u(t)$ which satisfies 
 \eqref{eqn:dODE}.
\end{lemma}

\begin{proof}
Lemma~\ref{lem:kappa} implies that \eqref{eqn:dODE} is 
 equivalent to 
$$u(t)=-\lambda \Imu u(t)+1 =: A_1 u.$$ 
Now the asymptotic and smoothness results of $\kappa(t)$ in Lemma~\ref{lem:kappa}
give $\kappa\in L^1(0,T)$, that is, there exists $t_1\in (0,T)$ such that 
$$\|\kappa\|_{L^1(0,t_1)}<\frac{1}{\lambda}.$$
 
 Hence, given $\phi_1, \phi_2\in L^1(0,t_1)$,  
 \begin{equation*}
  \begin{split}
    \|A_1(\phi_1)-A_1(\phi_2)\|_{L^1(0,t_1)}
 &\le \lambda \int_0^{t_1} \int_0^t |\kappa(t-s)|\cdot |\phi_1(s)-\phi_2(s)|\,
 ds dt\\
 &= \lambda\int_0^{t_1} |\phi_1(s)-\phi_2(s)| \int_s^{t_1} |\kappa(t-s)|\,
 dt ds\\
 &\le \lambda\int_0^{t_1} |\phi_1(s)-\phi_2(s)|\cdot \|\kappa\|_{L^1(0,t_1)} ds\\
 &= \lambda \|\kappa\|_{L^1(0,t_1)} \cdot\|\phi_1-\phi_2\|_{L^1(0,t_1)}.
  \end{split}
 \end{equation*}
From the fact that $0<\lambda\|\kappa\|_{L^1(0,t_1)}<1$,
$A_1$ is a contraction map on $L^1(0,t_1)$ and so by the
Banach fixed point theorem, there exists a 
unique $u_{1}(t)\in L^1(0,t_1)$ that satisfies $u_{1}=A_1u_{1}$.

For each $t\in (t_1,2t_1)$, we have 
\begin{equation*}
  u(t) = 1 -\lambda \Imu u(t)
= 1 -\lambda \int_{t_1}^{t} \kappa(t-s) u(s)\,ds
-\lambda \int_{0}^{t_1} \kappa(t-s) u(s)\,ds.
\end{equation*}
Since $u=u_{1}$ on $(0, t_1) $ which is now known, then 
\begin{equation*}
 \begin{split}
  u(t)=-\lambda \int_{t_1}^{t} \kappa(t-s) u(s)\,ds+1
-\lambda \int_{0}^{t_1} \kappa(t-s) u_1(s)\,ds:=A_2 u
 \end{split}
\end{equation*}
for each $t\in (t_1,2t_1)$.
Given $\phi_1, \phi_2\in L^1(t_1,2t_1)$, it holds 
\begin{equation*}
 \begin{split}
    \|A_2(\phi_1)-A_2(\phi_2)\|_{L^1(t_1,2t_1)}
    &\le \lambda \int_{t_1}^{2t_1} \int_{t_1}^{t} 
    |\kappa(t-s)|\cdot |\phi_1(s)-\phi_2(s)| ds dt\\
    &=\lambda \int_{t_1}^{2t_1} |\phi_1(s)-\phi_2(s)| 
    \int_{s}^{2t_1} |\kappa(t-s)| {\rm d} t {\rm d} s\\
    &\le \lambda \int_{t_1}^{2t_1} |\phi_1(s)-\phi_2(s)| 
    \cdot \|\kappa\|_{L^1(0,t_1)} {\rm d} s\\
    &=\lambda \|\kappa\|_{L^1(0,t_1)} \cdot
    \|\phi_1-\phi_2\|_{L^1(t_1,2t_1)}.
 \end{split}
\end{equation*}
Hence, $A_2$ is also a contraction map on $L^1(t_1,2t_1)$, 
which yields and shows that there exists a unique
$u_{2}(t)\in L^1(t_1,2t_1)$ such that $u_{2}=A_2 u_{2}$.

Repeating this argument yields that there exists a unique solution 
$u\in L^1(0,T)$ of the distributed ODE \eqref{eqn:dODE}, which completes 
the proof.
\end{proof}

\begin{lemma}\label{lem:u_is_cm}
$u(t)\in C^{\infty}(0,T)$ is completely monotone, which gives 
 $0\le u(t)\le 1$ on $[0,T]$.
\end{lemma}

\begin{proof}
This lemma is a special case of \cite[Theorem 2.3]{Kochubei:2008}.
\end{proof}

\section{Existence, uniqueness and regularity}

\subsection{Existence and uniqueness of weak solution for DDE \eqref{eqn:model_pde}
\label{sect:model_general}}

\par We state the definition of the weak solution as 
\begin{definition}
  $u(x,t)$ is a weak solution to 
   DDE \eqref{eqn:model_pde} in $L^2(\Omega)$ if 
   $u(\cdot,t)\in H_0^1(\Omega)$ for $t\in(0,T)$ and 
   for any $\psi(x)\in H^2(\Omega)\cap H_0^1(\Omega)$, 
   \begin{equation*}
   \begin{split}
    &\langle\Dmu u(x,t),\psi(x)\rangle-\langle\mathcal{L} u(x,t;a),\psi(x)\rangle=\langle f(x,t),\psi(x)\rangle,
   \ t\in(0,T);\\
   &\langle u(x,0),\psi(x)\rangle = \langle u_0(x),\psi(x)\rangle.
   \end{split}
   \end{equation*}
\end{definition}

\par Then Lemma \ref{lem:existence_uniqueness_u_n} gives the 
following corollary. 
\begin{corollary}\label{cor:existence_uniqueness}
 There exists a unique weak solution $u^*(x,t)$ of 
 DDE \eqref{eqn:model_pde} and the representation of $u^*(x,t)$ is 
\begin{equation}\label{eqn:weak solution}
\begin{aligned}
 u^*(x,t)=&\sum_{n=1}^{\infty} \Big[\langle u_0,\psi_n\rangle u_n(t)
 +\langle f(\cdot,0),\psi_n\rangle\Imu u_n(t)\\
 &+\int_0^t \langle\frac{\partial}{\partial t}
 f(\cdot,\tau),\psi_n\rangle \Imu u_n(t-\tau){\rm d}\tau\Big]\psi_n(x),
 \end{aligned}  
\end{equation}
where $u_n(t)$ is the unique solution of the distributed ODE \eqref{eqn:dODE} 
with $\lambda=\lambda_n$.
\end{corollary}
\begin{proof}
  \par Completeness of $\{\psi_n(x):\N+\}$ in $L^2(\Omega)$ and 
  direct calculation show that the representation 
  \eqref{eqn:weak solution} is a weak solution of DDE \eqref{eqn:model_pde}; 
  while the uniqueness of $u^*$ follows from Lemma \ref{lem:existence_uniqueness_u_n}.
\end{proof}

\subsection{Regularity}
\par The next two lemmas concern the regularity of 
$u^*$ and $\Dmu u^*$.
\begin{lemma}\label{lem:regularity_u}
 \begin{equation*}
  \|u^*(x,t)\|_{C([0,T];H^2(\Omega))}
\le C\big(\|u_0\|_{H^2(\Omega)}+\|f(\cdot, 0)\|_{H^2(\Omega)}
+T^{1/2}  |f|_{H^1([0,T];H^2(\Omega))}\big)
\end{equation*}
where $C>0$ depends on $\mu$, $\mathcal{L}$ and $\Omega$, 
and $|f|_{H^1([0,T];H^2(\Omega))}=
\|\frac{\partial f}{\partial t}\|_{L^2([0,T];H^2(\Omega))}$. 
\end{lemma}
\begin{proof}
 \par Fix $t\in(0,T)$,
 \begin{equation*}
 \begin{aligned}
  \|u^*(x,t)\|_{H^2(\Omega)}
 \le &\ \big\|\sum_{n=1}^{\infty}\langle u_0,\psi_n\rangle u_n(t)\psi_n(x)\big\|_{H^2(\Omega)}
 &&:=I_1\\
  &+\big\|\sum_{n=1}^{\infty}\langle f(\cdot,0),\psi_n\rangle\Imu u_n(t)\psi_n(x)\big\|_{H^2(\Omega)}
  &&:=I_2\\
 &+\big\|\sum_{n=1}^{\infty}\int_0^t \langle\frac{\partial}{\partial t}
 f(\cdot,\tau),\psi_n\rangle \Imu u_n(t-\tau){\rm d}\tau\ \psi_n(x)\big\|_{H^2(\Omega)}
 &&:=I_3. 
 \end{aligned}
\end{equation*}
We estimate each of $I_1$, $I_2$, and $I_3$ in turn using Lemmas~\ref{lem:kappa} and \ref{lem:u_is_cm} 
where in each case $C>0$ is a generic constant that depends only on
$\mu$, $\mathcal{L}$ and $\Omega$.
\begin{equation*}\label{eqn:I_1}
 \begin{aligned}
  I_1^2&=\big\|\sum_{n=1}^{\infty}\langle u_0,\psi_n\rangle u_n(t)\psi_n(x)\big\|_{H^2(\Omega)}^2
  \le C\big\|\mathcal{L}\big(\sum_{n=1}^{\infty}\langle u_0,\psi_n\rangle 
  u_n(t)\psi_n(x)\big)\big\|_{L^2(\Omega)}^2\\
 &=C\big\|\sum_{n=1}^{\infty}\lambda_n\langle u_0,\psi_n\rangle u_n(t)\psi_n(x)\big\|_{L^2(\Omega)}^2 
 =C\sum_{n=1}^{\infty} \lambda_n^2\langle u_0,\psi_n\rangle^2 u^2_n(t)\\
 &\le C\sum_{n=1}^{\infty} \lambda_n^2\langle u_0,\psi_n\rangle^2 
 =C\big\|\mathcal{L} u_0\big\|_{L^2(\Omega)}^2\le C\|u_0\|_{H^2(\Omega)}^2.
 \end{aligned}
 \end{equation*}
 \begin{equation*}
 \begin{aligned}
  I_2^2&=\big\|\sum_{n=1}^{\infty}\langle f(\cdot,0),\psi_n\rangle
  \Imu u_n(t)\psi_n(x)\big\|_{H^2(\Omega)}^2
  \le C\sum_{n=1}^{\infty}\lambda_n^2\langle f(\cdot,0),\psi_n\rangle^2 
  (\Imu u_n(t))^2\\
  &\le C\sum_{n=1}^{\infty}\lambda_n^2\langle f(\cdot,0),\psi_n\rangle^2 
  \Bigl(\int_0^t|\kappa(\tau)|
  \cdot|u_n(t-\tau)|{\rm d}\tau\Bigr)^2 \\
  &\le C \sum_{n=1}^{\infty}\lambda_n^2\langle f(\cdot,0),\psi_n\rangle^2 
  \Bigl(\int_0^t |\kappa(\tau)|{\rm d}\tau\Bigr)^2\\
  &\le C\sum_{n=1}^{\infty}\lambda_n^2\langle f(\cdot,0),\psi_n\rangle^2 \|\kappa\|_{L^1(0,T)}^2
  \le C\|\kappa\|_{L^1(0,T)}^2 \|f(\cdot, 0)\|_{H^2(\Omega)}^2.
 \end{aligned}
\end{equation*}
\begin{equation*}\label{eqn:I_3}
 \begin{aligned}
 \quad I_3^2&=\big\|\sum_{n=1}^{\infty}\int_0^t \langle\frac{\partial}{\partial t}
 f(\cdot,\tau),\psi_n\rangle \Imu u_n(t-\tau){\rm d}\tau\ \psi_n(x)\big\|_{H^2(\Omega)}^2\\
 &\le C\sum_{n=1}^{\infty}\left[\int_0^t \lambda_n \langle\frac{\partial}{\partial t}
 f(\cdot,\tau),\psi_n\rangle \Imu u_n(t-\tau){\rm d}\tau\right]^2\\
 &\le C\sum_{n=1}^{\infty}\left[\int_0^t \lambda_n|\langle\frac{\partial}{\partial t}
 f(\cdot,\tau),\psi_n\rangle|\cdot  |\Imu u_n(t-\tau)|{\rm d}\tau\right]^2\\
 &\le C\|\kappa\|_{L^1(0,T)}^2 \sum_{n=1}^{\infty}
 \int_0^t \lambda_n^2|\langle\frac{\partial}{\partial t} f(\cdot,\tau),\psi_n\rangle|^2 
 {\rm d}\tau \cdot \int_0^t 1^2 {\rm d}\tau\\
 &\le C T\|\kappa\|_{L^1(0,T)}^2  \int_0^T \sum_{n=1}^{\infty} 
 \lambda_n^2|\langle\frac{\partial}{\partial t} f(\cdot,\tau),\psi_n\rangle|^2{\rm d}\tau
 \le CT\|\kappa\|_{L^1(0,T)}^2  \int_0^T 
 \big\|\frac{\partial}{\partial t} f(\cdot,\tau)\big\|_{H^2(\Omega)}^2{\rm d}\tau\\
 &=CT\|\kappa\|_{L^1(0,T)}^2 |f|^2_{H^1([0,T];H^2(\Omega))}.
   \end{aligned}
\end{equation*}
 Hence, 
\begin{equation*}
 \begin{split}
  \|u^*(x,t)\|_{C([0,T];H^2(\Omega))}
&\le C\|u_0\|_{H^2( \Omega)}+C\|\kappa\|_{L^1(0,T)} \|f(\cdot, 0)\|_{H^2(\Omega)}\\
&\qquad+CT^{1/2}\|\kappa\|_{L^1(0,T)} |f|_{H^1([0,T];H^2(\Omega))}\\
&\le C\big(\|u_0\|_{H^2(\Omega)}+\|f(\cdot, 0)\|_{H^2(\Omega)}
+T^{1/2} |f|_{H^1([0,T];H^2(\Omega))}\big).
 \end{split}
\end{equation*}
Due to the fact that $\kappa$ is determined by $\mu$,
the constant $C$ above only depends on $\mu$, $\mathcal{L}$ and $\Omega$.
\end{proof}

\begin{lemma}\label{lem:regularity_Dmu}
 \begin{equation*}
 \begin{split}
  \|\Dmu u^*\|_{C([0,T];L^2(\Omega))}
\le C\left(\|u_0\|_{H^2(\Omega)}+T^{1/2}|f|_{H^1([0,T];H^2(\Omega))}
+\|f\|_{C([0,T];H^2(\Omega))}\right),
 \end{split}
\end{equation*}
where $C>0$ only depends on $\mu$, $\mathcal{L}$ and $\Omega$.
\end{lemma}
\begin{proof}
\par For each $t\in(0,T)$,
 \begin{equation*}
  \begin{split}
   \Dmu u^*(x,t)&=-\sum_{n=1}^{\infty}\lambda_n\langle u_0,\psi_n\rangle u_n(t)\psi_n(x)
 -\sum_{n=1}^{\infty}\lambda_n\langle f(\cdot,0),\psi_n\rangle\Imu u_n(t)\psi_n(x)\\
 &\quad-\sum_{n=1}^{\infty}\lambda_n\int_0^t \langle\frac{\partial}{\partial t}
 f(\cdot,\tau),\psi_n\rangle \Imu u_n(t-\tau){\rm d}\tau\ \psi_n(x)
 +f(x,t),
  \end{split}
 \end{equation*}
which implies
\begin{equation*}
 \begin{aligned}
 \qquad\|\Dmu u^*\|_{L^2(\Omega)}& \le 
 \|\sum_{n=1}^{\infty}\lambda_n\langle u_0,\psi_n\rangle u_n(t)\psi_n(x)\|_{L^2(\Omega)}
  +\|\sum_{n=1}^{\infty}\lambda_n\langle f(\cdot,0),\psi_n\rangle\Imu u_n(t)\psi_n(x)\|_{L^2(\Omega)}\\
 &\quad+\|\sum_{n=1}^{\infty}\lambda_n\int_0^t \langle\frac{\partial}{\partial t}
 f(\cdot,\tau),\psi_n\rangle \Imu u_n(t-\tau){\rm d}\tau\ \psi_n(x)\|_{L^2(\Omega)}
  +\|f(\cdot,t)\|_{L^2(\Omega)}.
 \end{aligned}
\end{equation*}
Combining the estimates for $I_1$, $I_2$ and $I_3$ we obtain
\begin{equation*}
 \|\sum_{n=1}^{\infty}\lambda_n\langle u_0,\psi_n\rangle u_n(t)\psi_n(x)\|_{L^2(\Omega)}^2
 =\sum_{n=1}^{\infty} \lambda_n^2\langle u_0,\psi_n\rangle^2u_n^2(t)
 \le C\|u_0\|_{H^2(\Omega)}^2,
\end{equation*}
\begin{equation*}
\begin{split}
 \|\sum_{n=1}^{\infty}\lambda_n\langle f(\cdot,0),\psi_n\rangle
 \Imu u_n(t)\psi_n(x)\|_{L^2(\Omega)}^2
 &=\sum_{n=1}^{\infty}\lambda_n^2\langle f(\cdot,0),\psi_n\rangle^2 (\Imu u_n(t))^2\\
 &\le C\|\kappa\|_{L^1(0,T)}^2 \|f(\cdot, 0)\|_{H^2(\Omega)}^2\\
 &\le C\|\kappa\|_{L^1(0,T)}^2 \|f\|_{C([0,T];H^2(\Omega))}^2
 \end{split} 
\end{equation*}
and 
\begin{equation*}
\begin{split}
&\|\sum_{n=1}^{\infty}\lambda_n\int_0^t\langle\frac{\partial}{\partial t}
 f(\cdot,\tau),\psi_n\rangle \Imu u_n(t-\tau){\rm d}\tau\ \psi_n(x)\|_{L^2(\Omega)}^2\\
=&\sum_{n=1}^{\infty}\left[\int_0^t \lambda_n\langle\frac{\partial}{\partial t}
 f(\cdot,\tau),\psi_n\rangle \Imu u_n(t-\tau){\rm d}\tau\right]^2
 \le CT\|\kappa\|_{L^1(0,T)}^2 |f|^2_{H^1([0,T];H^2(\Omega))}.
 \end{split} 
\end{equation*}
Therefore, 
\begin{equation*}
 \begin{split}
  \|\Dmu u^*\|_{C([0,T];L^2(\Omega))}
\le C\left(\|u_0\|_{H^2(\Omega)}+T^{1/2}|f|_{H^1([0,T];H^2(\Omega))}
+\|f\|_{C([0,T];H^2(\Omega))}\right),
 \end{split}
\end{equation*}
where $C$ is dependent only on $\mu$, $\mathcal{L}$ and $\Omega$.
\end{proof}

The main theorem of this section follows from
Corollary~\ref{cor:existence_uniqueness}, 
Lemmas~\ref{lem:regularity_u} and \ref{lem:regularity_Dmu}. 
\begin{theorem}[Main theorem for the direct problem]\label{main}
 There exists a unique weak solution $u^*(x,t)$ in $L^2(\Omega)$ of the 
 DDE~\eqref{eqn:model_pde} with the representation \eqref{eqn:weak solution} 
 and the following regularity estimate
 \begin{equation*}
  \begin{aligned}
   \quad\|u^*\|_{C([0,T];H^2(\Omega))} &+ \|\Dmu u^*\|_{C([0,T];L^2(\Omega))}\\
 &\le C\Big(\|u_0\|_{H^2(\Omega)}+T^{1/2}|f|_{H^1([0,T];H^2(\Omega))}
+\|f\|_{C([0,T];H^2(\Omega))}\Big),
  \end{aligned}
 \end{equation*}
where $C>0$ depends only on $\mu$, $\mathcal{L}$ and $\Omega$.
\end{theorem}

\section{Representation of the DDE solution for one spatial variable}
\label{sect:representation}

In this section, we will establish a representation result for the 
special case $\Omega=(0,1)$, $\mathcal{L}u=u_{xx}$ in  \eqref{eqn:model_pde}
 \begin{equation}\label{eqn:one_dim_model}
  \begin{cases}
   \Dmu u-u_{xx}=f(x,t),\ 0<x<1,\ 0<t<\infty;\\
   u(x,0)=u_0(x),\ 0<x<1;\\
   u(0,t)=g_0(t),\ 0\le t< \infty;\\
   u(1,t)=g_1(t),\ 0\le t< \infty,
  \end{cases}
 \end{equation}
 where $g_0,g_1\in L^2(0,\infty)$ and $f(x,\cdot)\in L^1(0,\infty)$ for each $x\in(0,1)$. 
  
We can obtain the fundamental solution  by Laplace and Fourier transforms.
First, we extend the finite domain to an infinite one and impose a homogeneous
right-hand side, i.e. we consider the following model
 \begin{equation*}
 \begin{cases}
 \Dmu u-u_{xx}=0,\ -\infty<x<\infty,\ 0<t<\infty;\\
 u(x,0)=u_0(x),\ -\infty<x<\infty.
 \end{cases}
 \end{equation*}
Next we take the Fourier transform $\FT$
with respect to $x$ and denote 
$(\FT u)(\xi,t)$ by $\tilde{u}(\xi,t)$,  
\begin{equation*}
 \Dmu \tilde{u}(\xi,t) +\xi^2\tilde{u}(\xi,t)=0.
\end{equation*}
Then by taking the Laplace transform $\L$ with respect to $t$ and denote 
$(\L \tilde{u})(\xi,z)$ by $\hat{\tilde{u}}(\xi,z)$,
we obtain
\begin{equation*}
 \int_0^1 \mu(\alpha)\left(z^\alpha\hat{\tilde{u}}(\xi,z)-
 z^{\alpha-1}\tilde{u}_0(\xi)\right)\,d\alpha
 +\xi^2\hat{\tilde{u}}(\xi,z)=0,
\end{equation*}
that is,
\begin{equation*}
 \hat{\tilde{u}}(\xi,z)=\frac{\Phi(z)/z}{\Phi(z)+\xi^2}\tilde{u}_0(\xi), 
\end{equation*}
where $\Phi(z)$ comes from \eqref{eqn:Phi}.

Then we have 
\begin{equation*}
 \begin{aligned}
  u(x,t)=\FT^{-1}\!\circ\L^{-1}( \hat{\tilde{u}}(\xi,z))
  &=\frac{1}{2\pi}\int_{-\infty}^{+\infty} e^{i x\xi}\frac{1}{2\pi i}
   \int_{\gamma-i\infty}^{\gamma+i\infty} e^{zt} 
   \frac{\Phi(z)/z}{\Phi(z)+\xi^2}\tilde{u}_0(\xi) \,dz \,d\xi\\
  &=\frac{1}{2\pi i}  \int_{\gamma-i\infty}^{\gamma+i\infty} e^{zt}
  \int_{-\infty}^{+\infty} \frac{1}{2\pi}e^{ i x\xi}\frac{\Phi(z)/z}{\Phi(z)+\xi^2}
  \tilde{u}_0(\xi) \,d\xi\,dz\\
  &=\frac{1}{2\pi i}  \int_{\gamma-i\infty}^{\gamma+i\infty} e^{zt}
  \big(\FT^{-1}(\frac{\Phi(z)/z}{\Phi(z)+\xi^2})*u_0\big)(x)\,dz,
 \end{aligned}
\end{equation*}
where the integral above is the usual Bromwich path, that is, a line in the complex plane parallel to the imaginary axis $z=\gamma + it$, $-\infty<t<\infty$, 
see \cite{WhittakerWatson:1962}. The last equality follows from the Fourier transform formula 
on convolutions and $\gamma$ can be an arbitrary positive number 
due to the fact that $z=0$ is a singular point of 
the function $\frac{\Phi(z)/z}{\Phi(z)+\xi^2}$. 
Throughout the remainder of this paper we will use $\gamma$ to denote a
strictly positive constant which is larger than $e^{1/\beta}$. The number 
$e^{1/\beta}$ will be seen in the proof of Lemma \ref{lem:phi}. 
We shall assume the angle of variation $z$ for the Laplace transforms is
from $-\pi$ to $\pi$, that is
$z\in\Lambda:=\{z\in \mathbb{C}:arg(z)\in (-\pi, \pi]\}$.

For $\Phi(z)$, we have the following result which will be central to
the rest of the paper.
It can be shown by using the Cauchy-Riemann equations in polar form. 
\begin{lemma}
 $\Phi(z)$ is analytic on $\mathbb{C}\setminus\!\{0\}$.
\end{lemma}

\par In the next two lemmas, we obtain important properties of $\Phi(z).$ 
\begin{lemma}\label{lem:re_phi}
$\;{\displaystyle
 \re(\Phi^{1/2}(z))\ge \frac{\sqrt{2}}{2}|\Phi^{1/2}(z)|,\ \re z=\gamma>0}$.
\end{lemma}
\begin{proof}
  $\gamma>0$ implies that $\re z>0$, i.e. 
  $arg(z)\in (-\frac{\pi}{2},\frac{\pi}{2})$, 
  which together with $0<\alpha<1$ and $\mu(\alpha)\ge 0$ yields 
  $\re \Phi(z)\ge 0$, i.e. $arg(\Phi(z))\in 
  (-\frac{\pi}{2},\frac{\pi}{2})$. This gives $arg(\Phi^{1/2}(z))\in 
  (-\frac{\pi}{4},\frac{\pi}{4})$. Hence, 
  $$\re(\Phi^{1/2}(z))=\cos(arg(\Phi^{1/2}(z)))|\Phi^{1/2}(z)|
  \ge\frac{\sqrt{2}}{2}|\Phi^{1/2}(z)|,$$ 
  which completes the proof.
\end{proof}

\begin{lemma}\label{lem:phi}
$$\;{\displaystyle
 C_{\mu, \beta}\frac{\gamma^\beta-\gamma^{\beta_0}}{\ln \gamma} \le 
  C_{\mu, \beta} \frac{|z|^\beta-|z|^{\beta_0}}{\ln |z|} 
  \le|\Phi(z)|\le C\frac{|z|-1}{\ln |z|}},$$ for $z$ such that  
  $\re z=\gamma>e^{1/\beta}>0$.
\end{lemma}

\begin{proof}
For the right-hand side of the inequality, 
  $\mu(\alpha)\in C^1[0,1]$ obviously implies that there exists a 
  $C>0$ such that $|\mu(\alpha)|\le C$ on $[0,1]$. 
  Hence, 
  \begin{equation*}
   \begin{aligned}
     |\Phi(z)|\le \int_0^1 |\mu(\alpha)|\cdot |z|^\alpha \,d\alpha
     \le C \int_0^1 |z|^\alpha \,d\alpha=C\frac{|z|-1}{\ln|z|}.
   \end{aligned}
  \end{equation*}
  
  \par
For the left-hand side, write $z=r e^{i\theta}$. Since $\re z=\gamma>0$, 
  $\theta\in(-\frac{\pi}{2},\frac{\pi}{2})$, then
  \begin{equation*}
   \begin{aligned}
    |\Phi (z)|&\ge \re(\phi(z))=\int_0^1 \mu(\alpha) r^\alpha 
  \cos(\theta \alpha) \,d\alpha \\
  &\ge C_\mu \int_{\beta_0}^\beta r^\alpha \cos(\theta\alpha)\,d\alpha
   \ge C_\mu \cos(\beta\theta) \int_{\beta_0}^\beta r^\alpha \,d\alpha\\
  &\ge C_\mu \cos(\frac{\beta\pi}{2}) \int_{\beta_0}^\beta |z|^\alpha 
  \,d\alpha
  =C_{\mu, \beta} \frac{|z|^\beta-|z|^{\beta_0}}{\ln |z|}.
   \end{aligned}
  \end{equation*}
Recall $|z|\ge \gamma>e^{1/\beta}$, we have 
$\frac{|z|^\beta-|z|^{\beta_0}}{\ln |z|} 
\ge \frac{\gamma^\beta-\gamma^{\beta_0}}{\ln \gamma}$ due to the function 
$\frac{x^\beta-x^{\beta_0}}{\ln x} $ being increasing on the interval 
$(e^{1/\beta},+\infty)$.
\end{proof}

Now we are in a position to calculate the complex integral 
$\FT^{-1}\bigl(\frac{\Phi(z)/z}{\Phi(z)+\xi^2}\bigr)$.

\begin{lemma}\label{lem:inversefourier}
$\;{\displaystyle
 \FT^{-1}(\frac{\Phi(z)/z}{\Phi(z)+\xi^2})
 =\frac{\Phi^{1/2}(z)}{2z}e^{-\Phi^{1/2}(z)|x|}}$.
\end{lemma}
\begin{proof}
From the inverse Fourier transform formula we have
\begin{equation*}
\begin{aligned}
 \FT^{-1}\Bigl(\frac{\Phi(z)/z}{\Phi(z)+\xi^2}\Bigr)
 =\frac{1}{2\pi}\int_{-\infty}^{+\infty} e^{i x\xi} 
 \frac{\Phi(z)/z}{\Phi(z)+\xi^2} \,d\xi.
 \end{aligned} 
\end{equation*}
We denote the contour from $-R$ to $R$ by 
$C_0$, the semicircle with radius $R$ in the upper and lower half plane 
by $C_{R^+}$ and $C_{R^{-}}$, respectively. 
Also, let $C_+$, $C_-$ be the closed contours which consist of 
$C_0, C_{R^+}$ and $C_0, C_{R^-}$ respectively.

For the case of $x>0$, working on the closed contour $C_+$, we have
\begin{equation*}
 \begin{aligned}
  \quad\frac{1}{2\pi}\int_{-\infty}^{+\infty}\!\!\!e^{ i x\xi} 
 \frac{\Phi(z)/z}{\Phi(z)+\xi^2} \,d\xi
 &=\lim_{R\to \infty} \frac{1}{2\pi}\oint_{C_+}\!\! e^{ i x\xi} 
 \frac{\Phi(z)/z}{\Phi(z)+\xi^2} \,d\xi
 -\lim_{R\to \infty} \frac{1}{2\pi}\int_{C_R^+} \!\!e^{ i x\xi} 
 \frac{\Phi(z)/z}{\Phi(z)+\xi^2} \,d\xi\\
 &=\lim_{R\to \infty} \frac{1}{2\pi}\oint_{C_+}\!\!e^{ i x\xi} 
 \frac{\Phi(z)/z}{\Phi(z)+\xi^2} \,d\xi,
 \end{aligned}
\end{equation*}
where the second limit is $0$ as follows from Jordan's Lemma.
Since to $0<\alpha<1$, $\gamma>0$, by our assumptions
we have $\re(\Phi(z))\ge 0$, which in turn leads to $\re(\Phi^{1/2}(z))\ge 0$. 
Then there is only one singular point $\xi=i\Phi^{1/2}(z)$ in 
$C_+$ which is contained by the upper half plane.
By the residue theorem \cite{WhittakerWatson:1962}, we have
\begin{equation*}
  \lim_{R\to \infty} \frac{1}{2\pi}\oint_{C_+}e^{ i x\xi} 
 \frac{\Phi(z)/z}{\Phi(z)+\xi^2} \,d\xi
=\lim_{R\to \infty} 2\pi i\frac{1}{2\pi} 
e^{ixi\Phi^{1/2}(z)} \frac{\Phi(z)/z}{2i\Phi^{1/2}(z)}
=\frac{\Phi^{1/2}(z)}{2z}e^{-\Phi^{1/2}(z)x}.
\end{equation*}
For the case of $x<0$, we choose the closed contour 
$C_-$.
Since $\re(\Phi^{1/2}(z))\ge 0$, it follows that 
$\xi=-i\Phi^{1/2}(z)$ is the unique singular point in $C_-$.
Then a similar calculation gives
\begin{equation*}
 \begin{aligned}
  \quad\frac{1}{2\pi}\int_{-\infty}^{+\infty}\!\!\! e^{ i x\xi} 
 \frac{\Phi(z)/z}{\Phi(z)+\xi^2} \,d\xi
 &=-\lim_{R\to \infty} \frac{1}{2\pi}\oint_{C_-}\!\! e^{ i x\xi} 
 \frac{\Phi(z)/z}{\Phi(z)+\xi^2} \,d\xi
 +\lim_{R\to \infty} \frac{1}{2\pi}\int_{C_R^-}\!\!  e^{ i x\xi} 
 \frac{\Phi(z)/z}{\Phi(z)+\xi^2} \,d\xi\\
 &=-\lim_{R\to \infty} \frac{1}{2\pi}\oint_{C_-}\!\! e^{ i x\xi} 
 \frac{\Phi(z)/z}{\Phi(z)+\xi^2} \,d\xi\\
 &=\lim_{R\to \infty} \frac{\Phi^{1/2}(z)}{2z}e^{\Phi^{1/2}(z)x}
 = \frac{\Phi^{1/2}(z)}{2z}e^{\Phi^{1/2}(z)x}.
 \end{aligned}
\end{equation*}

Therefore, 
$$
 \FT^{-1}\Bigl(\frac{\Phi(z)/z}{\Phi(z)+\xi^2}\Bigr)
 =\frac{\Phi^{1/2}(z)}{2z}e^{-\Phi^{1/2}(z)|x|},
$$
which completes the proof.
\end{proof}

\subsection{The fundamental solution $\,G_{\mu}(x,t)$}

With the above lemma, we have 
\begin{equation*}
 \begin{aligned}
  u(x,t)&=\frac{1}{2\pi i}  \int_{\gamma-i\infty}^{\gamma+i\infty} e^{zt}
  \int_{-\infty}^{+\infty} \frac{\Phi^{1/2}(z)}{2z}e^{-\Phi^{1/2}(z)|x-y|}
  u_0(y)\,dy\,dz\\
  &=\int_{-\infty}^{+\infty} \left[\frac{1}{2\pi i}
  \int_{\gamma-i\infty}^{\gamma+i\infty}\frac{\Phi^{1/2}(z)}{2z}
  e^{zt-\Phi^{1/2}(z)|x-y|}\,dz\right] u_0(y)\,dy.
 \end{aligned}
\end{equation*}
Then we can define the fundamental solution $G_{(\mu)}(x,t)$ as
\begin{equation}\label{G_mu}
 G_{(\mu)}(x,t)=\frac{1}{2\pi i}
  \int_{\gamma-i\infty}^{\gamma+i\infty}\frac{\Phi^{1/2}(z)}{2z}
  e^{zt-\Phi^{1/2}(z)|x|}\,dz.
\end{equation}

The following three lemmas provide some important properties of 
$G_{(\mu)}(x,t)$. 

\begin{lemma}\label{lem:pointwise}
 The integral for $G_{(\mu)}(x,t)$ is convergent for each 
 $(x,t)\in (0,\infty)\times(0,\infty)$.
\end{lemma}
\begin{proof}
 \par Given $(x,t)\in (0,\infty)\times(0,\infty)$, with Lemmas 
 \ref{lem:re_phi} and \ref{lem:phi}, we have
 \begin{equation*}
 \begin{aligned}
   |G_{(\mu)}(x,t)|
   &\le \frac{1}{4\pi} \int_{\gamma-i\infty}^{\gamma+i\infty}
   |\frac{\Phi^{1/2}(z)}{z}|\cdot|e^{zt}|\cdot|e^{-\Phi^{1/2}(z)|x|}|\,dz\\
   &= \frac{1}{4\pi} \int_{\gamma-i\infty}^{\gamma+i\infty}
   \frac{|\Phi^{1/2}(z)|}{|z|}e^{\gamma t} 
   e^{-\re(\Phi^{1/2}(z)|x|)} \,d z\\
   &\le \frac{1}{4\pi} \int_{\gamma-i\infty}^{\gamma+i\infty}
   \frac{|\Phi^{1/2}(z)|}{|z|}e^{\gamma t} 
   e^{-\frac{\sqrt{2}}{2}|x||\Phi^{1/2}(z)|} \,d z\\
   &\le\frac{Ce^{\gamma t}}{4\pi}\int_{\gamma-i\infty}^{\gamma+i\infty}
   (|z|\ln|z|)^{-1/2}
   e^{-C_{\mu, \beta}|x|(\frac{|z|^\beta-|z|^{\beta_0}}{\ln |z|})^{1/2}}
   \,dz\\
   &\le \frac{Ce^{\gamma t}}{4\pi (\ln\gamma)^{1/2}}
   \int_{\gamma-i\infty}^{\gamma+i\infty} |z|^{-1/2}
   e^{-C_{\mu, \beta}|x|(\frac{C|z|^\beta}{\ln |z|})^{1/2}}
   \,dz<\infty.
  \end{aligned}
\end{equation*}
\end{proof}

\begin{lemma}\label{eqn:G_smooth}
 $G_{(\mu)}(x,t)\in C^\infty((0,\infty)\times(0,\infty))$. 
\end{lemma}
\begin{proof}
Fix $(x,t)\in (0,\infty)\times(0,\infty)$.
Then for small $|\epsilon_x|, |\epsilon_t|$ we have
\begin{equation*}
 \begin{aligned}
 \quad |G_{(\mu)}(x+\epsilon_x,t+\epsilon_t)-G_{(\mu)}(x,t)|
 &\le |G_{(\mu)}(x+\epsilon_x,t+\epsilon_t)-G_{(\mu)}(x,t+\epsilon_t)|\\
 &\quad+ |G_{(\mu)}(x,t+\epsilon_t)-G_{(\mu)}(x,t)|.
  \end{aligned}
\end{equation*}
For $|G_{(\mu)}(x+\epsilon_x,t+\epsilon_t)-G_{(\mu)}(x,t+\epsilon_t)|$, 
 the following holds 
 \begin{equation*}
 \begin{aligned}
 &\quad|G_{(\mu)}(x+\epsilon_x,t+\epsilon_t)-G_{(\mu)}(x,t+\epsilon_t)|\\
 &\le \frac{1}{2\pi} \int_{\gamma-i\infty}^{\gamma+i\infty}
   |\frac{\Phi^{1/2}(z)}{2z}|\cdot|e^{zt+z\epsilon_t}|
   \cdot|e^{-\Phi^{1/2}(z)|x/2|}|\cdot
   |e^{-\Phi^{1/2}(z)(\frac{x}{2}+\epsilon_x)}
   -e^{-\Phi^{1/2}(z)(x/2)}|\ \,dz.
 \end{aligned} 
 \end{equation*}
From the proof of Lemma~\ref{lem:pointwise}, we have 
 \begin{equation*}
  \begin{aligned}
   |e^{-\Phi^{1/2}(z)(\frac{x}{2}+\epsilon_x)}-e^{-\Phi^{1/2}(z)(x/2)}|
 &\le |e^{-\Phi^{1/2}(z)(\frac{x}{2}+\epsilon_x)}|+|e^{-\Phi^{1/2}(z)(x/2)}|\\
 &\le e^{-\frac{\sqrt{2}}{2}|\Phi^{1/2}(z)|(\frac{x}{2}+\epsilon_x)}
 +e^{-\frac{\sqrt{2}}{2}|\Phi^{1/2}(z)|(x/2)} \le 2,
  \end{aligned}
 \end{equation*}
 and 
 $$
 \frac{1}{2\pi} \int_{\gamma-i\infty}^{\gamma+i\infty}
   |\frac{\Phi^{1/2}(z)}{2z}|\cdot|e^{zt+z\epsilon_t}|
   \cdot|e^{-\Phi^{1/2}(z)|x/2|}|\ \,dz<\infty.
 $$
Hence, after setting $e_1(z,\epsilon_x)=|e^{-\Phi^{1/2}(z)(\frac{x}{2}+\epsilon_x)}
-e^{-\Phi^{1/2}(z)(x/2)}|,$  
we can apply Lebesgue's dominated convergent theorem to deduce that 
\begin{equation*}
 \begin{aligned}
  &\lim_{\epsilon_x\to 0}|G_{(\mu)}(x+\epsilon_x,t+\epsilon_t)
  -G_{(\mu)}(x,t+\epsilon_t)|\\
 \le& \lim_{\epsilon_x\to 0} 
  \frac{1}{2\pi} \int_{\gamma-i\infty}^{\gamma+i\infty}
   |\frac{\Phi^{1/2}(z)}{2z}|\cdot|e^{zt+z\epsilon_t}|
   \!\cdot\!|e^{-\Phi^{1/2}(z)|x/2|}|\!\cdot e_1(z,\epsilon_x)\ \,dz\\
 =& \frac{1}{2\pi} \int_{\gamma-i\infty}^{\gamma+i\infty}
   |\frac{\Phi^{1/2}(z)}{2z}|\!\cdot\!|e^{zt+z\epsilon_t}|
   \!\cdot\!|e^{-\Phi^{1/2}(z)|x/2|}|\!\cdot\!
   \lim_{\epsilon_x\to 0}e_1(z,\epsilon_x)\ \,dz=0.
 \end{aligned}
\end{equation*}
A similar argument also shows that 
$\lim_{\epsilon_t\to 0}|G_{(\mu)}(x,t+\epsilon_t)-G_{(\mu)}(x,t)|=0$.
From this we deduce that 
$\lim_{\epsilon_x,\ \epsilon_t\to 0}
|G_{(\mu)}(x+\epsilon_x,t+\epsilon_t)-G_{(\mu)}(x,t)|=0$,
which shows that $G_{(\mu)}(x,t)\in C((0,\infty)\times(0,\infty))$.

Similarly, following from the proof of Lemma~\ref{lem:pointwise} 
and the above limiting argument, we obtain
$$G_{(\mu)}(x,t) \in C^n((0,\infty)\times(0,\infty)),\ \N+,$$ 
which leads to $G_{(\mu)}(x,t)\in C^\infty((0,\infty)\times(0,\infty))$ 
and this completes the proof.
\end{proof}

\begin{lemma}\label{lem:delta}
 \begin{equation*}
  \lim_{t\to 0} G_{(\mu)}(x,t)=\delta (x).
 \end{equation*}
\end{lemma}

\begin{proof}
\par Fix $x\ne 0$, for each $t\in (0,\infty)$, 
\begin{equation*}
\left|\frac{\Phi^{1/2}(z)}{2z}\right|\cdot |e^{zt-\Phi^{1/2}(z)|x|}| 
\le e^{\gamma t}  \left|\frac{\Phi^{1/2}(z)}{2z}\right|\cdot|e^{-\Phi^{1/2}(z)|x|}|.
\end{equation*}
The proof of Lemma \ref{lem:pointwise} shows that 
$$\int_{\gamma-i\infty}^{\gamma+i\infty} \left|\frac{\Phi^{1/2}(z)}{2z}\right|
\cdot|e^{-\Phi^{1/2}(z)|x|}|<\infty,$$ 
then by dominated convergence theorem, we can deduce that 
\begin{equation}\label{eqn:equality_2}
 \begin{aligned}
  \lim_{t\to 0} G_{(\mu)}(x,t)&=
\lim_{t\to 0}\frac{1}{2\pi i} \int_{\gamma-i\infty}^{\gamma+i\infty}
\frac{\Phi^{1/2}(z)}{2z}  e^{zt-\Phi^{1/2}(z)|x|}\,dz\\
&=\frac{1}{2\pi i} \int_{\gamma-i\infty}^{\gamma+i\infty}
\frac{\Phi^{1/2}(z)}{2z}  \lim_{t\to 0}e^{zt-\Phi^{1/2}(z)|x|}\,dz\\
&=\frac{1}{2\pi i} \int_{\gamma-i\infty}^{\gamma+i\infty}
\frac{\Phi^{1/2}(z)}{2z}e^{-\Phi^{1/2}(z)|x|}\,dz,
 \end{aligned}
\end{equation}
for each $x\ne 0$. Let $z=\gamma+mi$, we have 
\begin{equation}\label{eqn:equality_1}
 \begin{aligned}
 \lim_{t\to 0} G_{(\mu)}(x,t) =\frac{1}{4\pi}
  \int_{-\infty}^{+\infty}\frac{\Phi^{1/2}(\gamma+mi)}{\gamma+mi}
  e^{-\Phi^{1/2}(\gamma+mi)|x|}\,dm.
 \end{aligned}
\end{equation}
Recalling the definition of the closed contour $C_-$ and
the proof of Lemma~\ref{lem:inversefourier}, we see the function 
$\frac{\Phi^{1/2}(\gamma+mi)}{\gamma+mi}e^{-\Phi^{1/2}(\gamma+mi)|x|}$
is analytic in $C_-$.
Then 
\begin{equation*}
\begin{aligned}
\int_{-\infty}^{+\infty}\frac{\Phi^{1/2}(\gamma+mi)}{\gamma+mi}
  e^{-\Phi^{1/2}(\gamma+mi)|x|}\,dm
  &=\lim_{R\to \infty}\int_{C_{R^-}}\!\!\!\frac{\Phi^{1/2}(\gamma+mi)}{\gamma+mi}
  e^{-\Phi^{1/2}(\gamma+mi)|x|}\,dm\\
  &=\lim_{R\to \infty}\int_{-\pi}^0 Rie^{i\theta}\frac{\Phi^{1/2}(\gamma+Rie^{i\theta})}
  {\gamma+Rie^{i\theta}} e^{-\Phi^{1/2}(\gamma+Rie^{i\theta})|x|}\,d\theta,
\end{aligned}
\end{equation*}
where $m=Re^{i\theta}$.
Since $\re (\gamma+Rie^{i\theta})=\gamma-R\sin\theta\ge 0$, following from 
the proofs of Lemmas~\ref{lem:re_phi} and \ref{lem:phi}, we can deduce that 
\begin{equation*}
 \begin{aligned}
\re (\Phi^{1/2}(\gamma+Rie^{i\theta}))
&\ge \frac{\sqrt{2}}{2} |\Phi^{1/2}(\gamma+Rie^{i\theta})|\\
&\ge C_{\mu, \beta} \frac{|\gamma+Rie^{i\theta}|^\beta
-|\gamma+Rie^{i\theta}|^{\beta_0}}{\ln |\gamma+Rie^{i\theta}|}
\ge C \frac{R^\beta-R^{\beta_0}}{\ln R},
 \end{aligned}
\end{equation*}
and 
$$|\Phi^{1/2}(\gamma+Rie^{i\theta})|
\le C\frac{|\gamma+Rie^{i\theta}|-1}{\ln |\gamma+Rie^{i\theta}|}
\le C\frac{|R|-1}{\ln |R|}$$
for large $R$.
Hence, as $R\to \infty$,  
\begin{equation*}
 \begin{aligned}
\Bigl|Rie^{i\theta}\frac{\Phi^{1/2}(\gamma+Rie^{i\theta})}
  {\gamma+Rie^{i\theta}}&e^{-\Phi^{1/2}(\gamma+Rie^{i\theta})|x|}\Bigr|\\
  &\le |\frac{Rie^{i\theta}}{\gamma+Rie^{i\theta}}|
  \!\cdot\! |\Phi^{1/2}(\gamma+Rie^{i\theta})|
  \!\cdot\! |e^{-\Phi^{1/2}(\gamma+Rie^{i\theta})|x|}|\\
  &\le C\frac{|R|-1}{\ln |R|}\!\cdot\! e^{-C \frac{R^\beta-R^{\beta_0}}{\ln R}|x|}
  \to 0,
 \end{aligned}
\end{equation*}
which implies 
$$\left|\int_{-\infty}^{+\infty}\frac{\Phi^{1/2}(\gamma+mi)}{\gamma+mi}
  e^{-\Phi^{1/2}(\gamma+mi)|x|}\,dm\right|
  \le \pi\cdot C\frac{|R|-1}{\ln |R|}\cdot e^{-C \frac{R^\beta-R^{\beta_0}}{\ln R}|x|}
  \to 0.$$
The above result and \eqref{eqn:equality_1} show that 
\begin{equation}\label{eqn:equality_3}
 \lim_{t\to 0} G_{(\mu)}(x,t)=0\ \text{for}\ x\ne 0.
\end{equation}

Now, we are in the position to calculate 
$\int_{-\infty}^\infty \lim_{t\to 0} G_{(\mu)}(x,t) \,dx$.
Equation~\eqref{eqn:equality_2} gives
\begin{equation*}
 \begin{aligned}
\int_{-\infty}^\infty \lim_{t\to 0} G_{(\mu)}(x,t) \,dx
  &=\int_{-\infty}^0 \lim_{t\to 0} G_{(\mu)}(x,t) \,dx
  +\int_0^\infty \lim_{t\to 0} G_{(\mu)}(x,t) \,dx\\
  &=\int_{-\infty}^0 \frac{1}{2\pi i} \int_{\gamma-i\infty}^{\gamma+i\infty}
\frac{\Phi^{1/2}(z)}{2z}e^{-\Phi^{1/2}(z)|x|}\,dz \,dx\\
&\qquad+\int_0^\infty \frac{1}{2\pi i} \int_{\gamma-i\infty}^{\gamma+i\infty}
\frac{\Phi^{1/2}(z)}{2z}e^{-\Phi^{1/2}(z)|x|}\,dz \,dx\\
  &=\frac{1}{2\pi i}  \int_{\gamma-i\infty}^{\gamma+i\infty}\int_{-\infty}^0
\frac{\Phi^{1/2}(z)}{2z}e^{\Phi^{1/2}(z)x}\,dx \,dz\\
&\qquad+\frac{1}{2\pi i} \int_{\gamma-i\infty}^{\gamma+i\infty}\int_0^\infty 
\frac{\Phi^{1/2}(z)}{2z}e^{-\Phi^{1/2}(z)x}\,dx \,dz.
\end{aligned}
\end{equation*}
Now Lemma~\ref{lem:re_phi} and the fact that $\re z=\gamma>0$ shows that 
\begin{equation*}
 \begin{aligned}
&\int_{-\infty}^0\frac{\Phi^{1/2}(z)}{2z}e^{\Phi^{1/2}(z)x}\,dx
=\frac{e^{\Phi^{1/2}(z)x}}{2z}\Big |_{-\infty}^0=\frac{1}{2z},\\
&\int_{0}^\infty\frac{\Phi^{1/2}(z)}{2z}e^{-\Phi^{1/2}(z)x}\,dx
=\frac{e^{-\Phi^{1/2}(z)x}}{2z}\Big |_\infty^0=\frac{1}{2z}.\\
 \end{aligned}
\end{equation*}
Therefore, 
${\displaystyle\;
\int_{-\infty}^\infty \lim_{t\to 0} G_{(\mu)}(x,t) \,dx
=\frac{1}{2\pi i}  \int_{\gamma-i\infty}^{\gamma+i\infty}
\frac{1}{2z}\cdot 2\,dz=1
}$,
which together with \eqref{eqn:equality_3} yields the conclusion.
\end{proof}
Lemma~\ref{lem:delta} allows us to make the definition
\begin{equation}\label{initial_G_mu}
G_{(\mu)}(x,0)=\lim_{t\to 0} G_{(\mu)}(x,t)=\delta (x).
\end{equation}

\subsection{The Theta functions:
$\theta_{\mu}(x,t)$ and $\overline{\theta}_{\mu}(x,t)$}

One very useful way to represent solutions to initial value problems
for a parabolic equation is through the $\theta-$function, \cite{Cannon:1984}.
For the case of the heat equation if we let $K(x,t)$ denote the
fundamental solution, then set
$\theta(x,t) = \sum_{m=-\infty}^{\infty}K(x+2m,t)$.
The value of this function lies in the following result.
If $u_t-u_{xx}=0$, $u(0,t)=f_0(t)$, $u(1,t)=f_1(t)$, $u(x,0)=u_0(x)$, then
$u(x,t)$ has the representation
\begin{equation}
\begin{aligned}
u(x,t) &= \int_0^1[\theta(x-\xi,t)-\theta(x+\xi,t)]u_0(\xi)\,d\xi\\
&\quad -2\int_0^t \frac{\partial\theta}{\partial x}(x,t-\tau)f_0(\tau)\,d\tau
+2\int_0^t \frac{\partial\theta}{\partial x}(x-1,t-\tau)f_1(\tau)\,d\tau.
\end{aligned}
\end{equation}
A generalization to the case of the fractional equation
$D_t^\alpha -u_{xx} = 0$ for a fixed $\alpha$, $0<\alpha\leq 1$ can be found
in \cite{RundellXuZuo:2013}.
Our aim is to
extend this representation result to the distributed fractional order case.

\begin{definition}\label{def:theta_func}
We define for each $\mu(\alpha)$ which satisfies Assumption \ref{mu_assumption},   
$${\displaystyle\;
\theta_{(\mu)}(x,t)=\sum_{m=-\infty}^{\infty} G_{(\mu)}(x+2m,t)}.$$
\end{definition}

\par The uniform convergence and smoothness property 
of $\theta_{(\mu)}(x,t)$ are established by the next lemma. 
\begin{lemma}\label{lem:uniform_theta}
 $\theta_{(\mu)}(x,t)$ is an even function on $x$ and 
 uniformly convergent on $(0,2)\times (0,T)$ for any positive $T$.
 Then $\theta_{(\mu)}(x,t)\in C^\infty((0,2)\times (0,\infty))$.
\end{lemma}

\begin{proof}
The even symmetric property follows from the definitions 
of $G_{(\mu)}(x,t)$ and $\theta_{(\mu)}(x,t)$ directly.

\par Given a positive $T,$ fix $(x,t)\in (0,2)\times (0,T)$, by Lemma \ref{lem:re_phi} we have
\begin{equation}\label{leftsum}
\begin{aligned}
  \sum_{|m|>N}|G_{(\mu)}(x+2m,t)|
  &\le \left|\frac{1}{2\pi i} \sum_{|m|>N}\int_{\gamma-i\infty}^{\gamma+i\infty}
 \frac{\Phi^{1/2}(z)}{2z} e^{zt-\Phi^{1/2}(z)|x+2m|} \,dz\right|\\
 &= \left|\frac{1}{2\pi i} \int_{\gamma-i\infty}^{\gamma+i\infty}
 \frac{\Phi^{1/2}(z)}{2z} \sum_{|m|>N}e^{zt-\Phi^{1/2}(z)|x+2m|} \,dz\right|\\
 &\le \frac{1}{2\pi } \int_{\gamma-i\infty}^{\gamma+i\infty} 
  \big|\frac{\Phi^{1/2}(z)}{2z}\big|e^{\gamma t}\sum_{|m|>N}
  e^{-\re(\Phi^{1/2}(z))|x+2m|} \,dz\\
 &\le \frac{1}{2\pi } \int_{\gamma-i\infty}^{\gamma+i\infty} 
  \big|\frac{\Phi^{1/2}(z)}{2z}\big|e^{\gamma t}\sum_{|m|>N}
  e^{-\frac{\sqrt{2}}{2}|\Phi^{1/2}(z)||x+2m|} \,dz.
 \end{aligned}
\end{equation}
For the series $\sum_{|m|>N}
  e^{-\frac{\sqrt{2}}{2}|\Phi^{1/2}(z)||x+2m|}$, 
Lemma~\ref{lem:phi} shows that 
\begin{equation*}
\begin{aligned}
&\sum_{|m|>N}e^{-\frac{\sqrt{2}}{2}|\Phi^{1/2}(z)||x+2m|}\\
=&\ (1-e^{-\sqrt{2}|\Phi^{1/2}(z)|})^{-1}
(e^{-\frac{\sqrt{2}}{2}|\Phi^{1/2}(z)|(2N+2+x)}+
e^{-\frac{\sqrt{2}}{2}|\Phi^{1/2}(z)|(2N+2-x)})\\
=&\ \frac{e^{-\frac{\sqrt{2}}{2}|\Phi^{1/2}(z)|(2N-2)}}
{1-e^{-\sqrt{2}|\Phi^{1/2}(z)|}}
e^{-\frac{\sqrt{2}}{2}|\Phi^{1/2}(z)|}
(e^{-\frac{\sqrt{2}}{2}|\Phi^{1/2}(z)|(3+x)}+
e^{-\frac{\sqrt{2}}{2}|\Phi^{1/2}(z)|(3-x)})\\
\le&\ 2 (1-e^{-\sqrt{2}(C_{\mu, \beta}\frac{\gamma^\beta-\gamma^{\beta_0}}
{\ln \gamma})^{1/2}})^{-1}
(e^{-\frac{\sqrt{2}}{2}(C_{\mu, \beta}\frac{\gamma^\beta-\gamma^{\beta_0}}
{\ln \gamma})^{1/2}})^{2N-2}
e^{-\frac{\sqrt{2}}{2}|\Phi^{1/2}(z)|}\\
\le&\ A_\gamma C_\gamma^{2N-2}e^{-\frac{\sqrt{2}}{2}|\Phi^{1/2}(z)|}
\end{aligned}
\end{equation*}
where 
$$A_\gamma=2 (1-e^{-\sqrt{2}(C_{\mu, \beta}\frac{\gamma^\beta-\gamma^{\beta_0}}
{\ln \gamma})^{1/2}})^{-1}, 
\quad
0<C_\gamma=e^{-\frac{\sqrt{2}}{2}(C_{\mu,\beta}\frac{\gamma^\beta-\gamma^{\beta_0}}
{\ln \gamma})^{1/2}}<1$$ 
only depend on $\gamma>0$. 
Inserting the above result into \eqref{leftsum} yields 
$$ 
\sum_{|m|>N}|G_{(\mu)}(x+2m,t)| 
\le  \frac{1}{2\pi} \int_{\gamma-i\infty}^{\gamma+i\infty} 
  \big|\frac{\Phi^{1/2}(z)}{2z}\big|e^{\gamma t} 
  A_\gamma C_\gamma^{2N-2}e^{-\frac{\sqrt{2}}{2}|\Phi^{1/2}(z)|}\,dz.
$$
Meanwhile, from the proof of Lemma~\ref{lem:pointwise}, we have 
$$\int_{\gamma-i\infty}^{\gamma+i\infty}   \big|\frac{\Phi^{1/2}(z)}{2z}\big|
  e^{-\frac{\sqrt{2}}{2}|\Phi^{1/2}(z)|}\,dz<\infty.$$ 
Therefore, 
$$ 
\sum_{|m|>N}|G_{(\mu)}(x+2m,t)| 
\le  CC_\gamma^{2N-2}
$$
where the constant $C$ only depends on 
$T$, $\gamma$ and $0<C_\gamma<1$ only depends on $\gamma$. 
We conclude from this that for each $\epsilon>0$, $\exists$ sufficiently large 
$N\in\mathbb{N}$ independent of $x,t$ such that
$$
\sum_{|m|>N}|G_{(\mu)}(x+2m,t)|<\epsilon
\ \text{for each}\ (x,t)\in (0,2)\times (0,T),
$$
which implies the uniform convergence of the series. 
Then the smoothness results follow from Lemma~\ref{eqn:G_smooth}  
and the uniform convergence.
\end{proof}

\par Now we introduce the definition of $\overline{\theta}_{(\mu)}(x,t)$ 
and state some of its properties. 
\begin{definition}
 \begin{equation*}\label{eqn:theta_bar}
 \overline{\theta}_{(\mu)}(x,t) 
 =\left(\Imu \frac{\partial ^2 \theta_{(\mu)}}{\partial t \partial x}
 \right)(x,t),\ (x,t)\in (0,2)\times(0,\infty).
 \end{equation*}
\end{definition}

\begin{lemma}\label{lem:dmu_xx}
$\Dmu\theta_{(\mu)}(x,t)=(\theta_{(\mu)}(x,t))_{xx}$,\quad  
$\Dmu\overline{\theta}_{(\mu)}(x,t)=(\overline{\theta}_{(\mu)}(x,t))_{xx}$ .
\end{lemma}
\begin{proof}
 \par The first equality follows from the fact 
$\Dmu G_{(\mu)}(x,t)=(G_{(\mu)}(x,t))_{xx}$
and the uniform convergence of the series representation.
 
For the second equality, Lemma~\ref{lem:kappa} yields 
$\Dmu\overline{\theta}_{(\mu)}=\Dmu\Imu \frac{\partial ^2 
 \theta_{(\mu)}}{\partial t \partial x}
 =\frac{\partial ^2 \theta_{(\mu)}}{\partial t \partial x}$ and this together
with the first equality and Lemma~\ref{lem:uniform_theta} then gives
 \begin{equation*}
  \begin{aligned}
   (\overline{\theta}_{(\mu)})_{xx}
   &=\Imu \frac{\partial ^2 }{\partial t \partial x}(\frac{\partial^2\theta_{(\mu)}}{\partial x^2})
   =\Imu \frac{\partial ^2 }{\partial t \partial x} \Dmu\theta_{(\mu)}
   =\Imu \frac{\partial }{\partial t} \Dmu (\frac{\partial \theta_{(\mu)}}{\partial x})\\
   &=\kappa * \frac{\partial }{\partial t} 
   [\eta * \frac{\partial^2 \theta_{(\mu)}}{\partial t\partial x} ]
   =\kappa * \eta * \frac{\partial^3 \theta_{(\mu)}}{\partial t^2\partial x} 
   +\kappa * \eta\cdot \frac{\partial^2 \theta_{(\mu)}}{\partial t\partial x} (x,0)\\
   &=\int_0^t \frac{\partial^3 \theta_{(\mu)}}{\partial t^2 \partial x}  \,dt
   +\frac{\partial^2 \theta_{(\mu)}}{\partial t\partial x} (x,0)\\
   &=\frac{\partial^2 \theta_{(\mu)}}{\partial t\partial x} (x,t)
   -\frac{\partial^2 \theta_{(\mu)}}{\partial t\partial x} (x,0)
   +\frac{\partial^2 \theta_{(\mu)}}{\partial t\partial x} (x,0)
   =\frac{\partial^2  \theta_{(\mu)}}{\partial t\partial x},
  \end{aligned}
 \end{equation*}
which shows that the second equality holds. 
\end{proof}

\begin{lemma}\label{lem:boundary_overline_theta}
For each $\psi(t)\in L^2(0,\infty)$, we have
 \begin{equation*}
 \begin{split}
&\int_0^t \overline{\theta}_{(\mu)}(0+,t-s)\psi(s){\rm d}s=-\frac{1}{2} \psi(t),
\quad
\int_0^t \overline{\theta}_{(\mu)}(1-,t-s)\psi(s)\,ds = 0,\\
&\int_0^t \overline{\theta}_{(\mu)}(0-,t-s)\psi(s){\rm d}s=\frac{1}{2} \psi(t),
\quad
\int_0^t \overline{\theta}_{(\mu)}(-1+,t-s)\psi(s)\,ds = 0,\quad t\in (0,\infty).
 \end{split}   
 \end{equation*}
\end{lemma}

\begin{proof}
Fix $(x,t)\in (0,1)\times(0,\infty)$, then computing the Laplace 
transform yields 
\begin{equation}\label{eqn:L_theta}
\begin{aligned}
\L(\overline{\theta}_{(\mu)}(x,t))
   &=\L \Bigl[\kappa * \Bigl(\frac{\partial ^2}{\partial t\partial x}
   \sum_{m=-\infty}^{+\infty} G_{(\mu)}(x,t)\Bigr)\Bigr]\\
   &=\L \Bigl[\kappa *\Bigl(\sum_{m=-1}^{-\infty} \frac{1}{2\pi i}
   \int_{\gamma-i\infty}^{\gamma+i\infty}\frac{\Phi(z)}{2}
  e^{zt+\Phi^{1/2}(z)(x+2m)}\,dz\\
  &\quad-\sum_{m=0}^{+\infty} \frac{1}{2\pi i}
   \int_{\gamma-i\infty}^{\gamma+i\infty}\frac{\Phi(z)}{2}
  e^{zt-\Phi^{1/2}(z)(x+2m)}\,dz\Bigr)\Bigr]\\
  &=\L (\kappa)\cdot \L\Bigl(\sum_{m=-1}^{-\infty} \frac{1}{2\pi i}
   \int_{\gamma-i\infty}^{\gamma+i\infty}\frac{\Phi(z)}{2}
  e^{zt+\Phi^{1/2}(z)(x+2m)}\,dz\\
  &\quad-\sum_{m=0}^{+\infty} \frac{1}{2\pi i}
   \int_{\gamma-i\infty}^{\gamma+i\infty}\frac{\Phi(z)}{2}
  e^{zt-\Phi^{1/2}(z)(x+2m)}\,dz\Bigr)\\
  &=\frac{1}{\Phi(z)}\Bigl(\sum_{m=-1}^{-\infty} \frac{\Phi(z)}{2}e^{\Phi^{1/2}(z)(x+2m)}
  -\sum_{m=0}^{+\infty} \frac{\Phi(z)}{2}e^{-\Phi^{1/2}(z)(x+2m)}\Bigr)\\
  &=\frac{e^{(x-2)\Phi^{1/2}(z)}-e^{-x\Phi^{1/2}(z)}}
  {2(1-e^{-2\Phi^{1/2}(z)})},
  \end{aligned}
 \end{equation}
where the last equality follows from the fact 
$\re (\Phi^{1/2}(z))>0$ which is in turn ensured by Lemma~\ref{lem:re_phi}.
Therefore, 
 \begin{equation*}
 \begin{aligned}
 &\L\left(\int_0^t \overline{\theta}_{(\mu)}(0+,t-s)\psi(s) \,ds\right)
 =\L(\overline{\theta}_{(\mu)}(0+,t))\L(\psi(t))
 =-\frac{1}{2}\L(\psi(t));\\
 &\L\left(\int_0^t \overline{\theta}_{(\mu)}(1-,t-s)\psi(s) \,ds\right)
 =\L(\overline{\theta}_{(\mu)}(1-,t))\L(\psi(t))=0.
 \end{aligned} 
 \end{equation*}

 \par For $(x,t)\in (-1,0)\times(0,\infty)$, we have 
 \begin{equation*}
\begin{aligned}
\L(\overline{\theta}_{(\mu)}(x,t))
   &=\L \Bigl[\kappa *\Bigl(\sum_{m=0}^{-\infty} \frac{1}{2\pi i}
   \int_{\gamma-i\infty}^{\gamma+i\infty}\frac{\Phi(z)}{2}
  e^{zt+\Phi^{1/2}(z)(x+2m)}\,dz\\
  &\quad-\sum_{m=1}^{+\infty} \frac{1}{2\pi i}
   \int_{\gamma-i\infty}^{\gamma+i\infty}\frac{\Phi(z)}{2}
  e^{zt-\Phi^{1/2}(z)(x+2m)}\,dz\Bigr)\Bigr]\\
  &=\frac{1}{\Phi(z)}\Bigl(\sum_{m=0}^{-\infty} \frac{\Phi(z)}{2}e^{\Phi^{1/2}(z)(x+2m)}
  -\sum_{m=1}^{+\infty} \frac{\Phi(z)}{2}e^{-\Phi^{1/2}(z)(x+2m)}\Bigr)\\
  &=\frac{e^{x\Phi^{1/2}(z)}-e^{-(x+2)\Phi^{1/2}(z)}}
  {2(1-e^{-2\Phi^{1/2}(z)})},
  \end{aligned}
 \end{equation*}
 which gives 
 $\L(\overline{\theta}_{(\mu)}(0-,t))=\frac{1}{2}$ and  
 $\L(\overline{\theta}_{(\mu)}(-1+,t))=0,$ and completes the proof.
\end{proof}

\subsection{Representation of the solution to the initial-boundary value problem}

We will build the representation of the solution in this subsection from
four representations in terms of the theta functions;
the initial condition, the values of $u$ at each boundary $x=0$, $x=1$, and the
nonhomogeneous term $f$.

\begin{definition}\label{eqn:theta_kernels}
\begin{equation*}
 \begin{aligned}
u_1(x,t)&=\int_0^1(\theta_{(\mu)}(x-y,t)-\theta_{(\mu)}(x+y,t))u_0(y) 
   \,dy;\\
   u_2(x,t)&=-2\int_0^t \overline{\theta}_{(\mu)}(x,t-s)g_0(s)
   \,ds;\\
   u_3(x,t)&=2\int_0^t  \overline{\theta}_{(\mu)}(x-1,t-s)
   g_1(s)\,ds;\\
   u_4(x,t)&=\int_0^1\int_0^t[\theta_{(\mu)}(x-y,t-s)-
   \theta_{(\mu)}(x+y,t-s)]\cdot 
   [\frac{\partial}{\partial t}\Imu f(y,s)] \,ds\,dy.
  \end{aligned}
 \end{equation*}
\end{definition}

\par The following four lemmas give some properties of $u_j,\ j=1,2,3,4$.
\begin{lemma}\label{lem:u1u2u3u4}
$\;{\displaystyle \Dmu u_j=\frac{\partial^2  u_j}{\partial x^2},\ j=1,2,3}$,
$\;{\displaystyle 
\Dmu u_4=\frac{\partial^2 u_4}{\partial x^2}+f(x,t)}$,
where $(x,t)\in (0,1)\times(0,\infty)$.
\end{lemma}

\begin{proof}
For $u_1$, by Lemma~\ref{lem:dmu_xx}, we have
 \begin{equation*}
  \begin{aligned}
   \Dmu u_1 
   &= \int_0^1(\Dmu\theta_{(\mu)}(x-y,t)-\Dmu\theta_{(\mu)}(x+y,t))u_0(y) 
   \,dy\\
   &=\int_0^x(\Dmu\theta_{(\mu)}(x-y,t)-\Dmu\theta_{(\mu)}(x+y,t))u_0(y) 
   \,dy\\
   &\quad+
   \int_x^1(\Dmu\theta_{(\mu)}(x-y,t)-\Dmu\theta_{(\mu)}(x+y,t))u_0(y) 
   \,dy\\
   &=\int_0^x\Big[\theta_{(\mu)}(x-y,t)-\theta_{(\mu)}(x+y,t)\Big]_{xx}u_0(y) 
   \,dy\\
   &\quad+\int_x^1\Big[\theta_{(\mu)}(x-y,t)-\theta_{(\mu)}(x+y,t)\Big]_{xx}u_0(y) 
   \,dy\\
   &=\int_0^1\Big[\theta_{(\mu)}(x-y,t)-\theta_{(\mu)}(x+y,t)\Big]_{xx}u_0(y) 
   \,dy= \frac{\partial^2 u_1}{\partial x^2}.
  \end{aligned}
 \end{equation*}
 
For $u_2$, 
 \begin{equation*}
  \begin{aligned}
   \Dmu u_2 &=\eta * \frac{\partial u_2}{\partial t}
   =-2\eta * \frac{\partial}{\partial t}(\overline{\theta}_{(\mu)}*g_0)
   =-2\eta *(\frac{\partial}{\partial t}\overline{\theta}_{(\mu)})*g_0
   -2(\eta *g_0)\cdot \overline{\theta}_{(\mu)}(x,0)\\
   &=-2\Dmu \overline{\theta}_{(\mu)} *g_0
   =-2(\overline{\theta}_{(\mu)})_{xx}*g_0
   =(-2\overline{\theta}_{(\mu)}*g_0)_{xx}
   =(u_2)_{xx}.
  \end{aligned}
 \end{equation*}
In an analogous fashion to the above argument, we deduce that
$\Dmu u_3=(u_3)_{xx}$.

For $u_4$, using Lemmas~\ref{lem:delta}, \ref{lem:kappa} and
\ref{lem:uniform_theta} we obtain
\begin{equation*}
 \begin{aligned}
\Dmu u_4&=\eta * \frac{\partial u_4}{\partial t}
=\eta * \frac{\partial}{\partial t} 
\Bigl(\int_0^1 [\theta_{(\mu)}(x-y,\cdot)-\theta_{(\mu)}(x+y,\cdot)]
* [\frac{\partial}{\partial t}\Imu f(y,\cdot)]\,dy\Bigr)\\
&=\eta *\Bigl(\int_0^1 \frac{\partial}{\partial t}[\theta_{(\mu)}(x-y,\cdot)-\theta_{(\mu)}(x+y,\cdot)]
* [\frac{\partial}{\partial t}\Imu f(y,\cdot)]\,dy\Bigr)\\
&\quad+ \eta *\Bigl(\int_0^1 [\theta_{(\mu)}(x-y,0)-\theta_{(\mu)}(x+y,0)]
\cdot[\frac{\partial}{\partial t}\Imu f(y,t)]\,dy\Bigr)\\
&=\int_0^1 \eta*\frac{\partial}{\partial t}[\theta_{(\mu)}(x-y,\cdot)-\theta_{(\mu)}(x+y,\cdot)]
* [\frac{\partial}{\partial t}\Imu f(y,\cdot)]\,dy\\
&\quad +\eta *\Bigl(\int_0^1 [\delta(x-y)-\delta(x+y)]
\cdot[\frac{\partial}{\partial t}\Imu f(y,t)]\,dy\Bigr)\\
&=\int_0^1\Dmu [\theta_{(\mu)}(x-y,\cdot)-\theta_{(\mu)}(x+y,\cdot)]
* [\frac{\partial}{\partial t}\Imu f(y,\cdot)]\,dy
+\eta *\frac{\partial}{\partial t}\Imu f(x,t)\\
&=\int_0^1[\theta_{(\mu)}(x-y,\cdot)-\theta_{(\mu)}(x+y,\cdot)]_{xx}
* [\frac{\partial}{\partial t}\Imu f(y,\cdot)]\,dy
+\Dmu\Imu f(x,t)\\
&=(u_4)_{xx}+f(x,t).
 \end{aligned}
\end{equation*}
\end{proof}

\begin{lemma}\label{lem:initial_u}
$\;{\displaystyle \lim_{t\to 0} u_1(x,t)=u_0(x)}$,
$\;{\displaystyle \lim_{t\to 0} u_j(x,t)=0}$
for $j=2,3,4$, $x\in (0,1)$.
\end{lemma}

\begin{proof}
\par For each $x\in (0,1)$, Lemmas~\ref{lem:uniform_theta} and 
\ref{initial_G_mu} yield that 
\begin{equation*}
\begin{aligned}
 \lim_{t\to 0} u_1&= \int_0^1(\theta_{(\mu)}(x-y,0)-\theta_{(\mu)}(x+y,0))u_0(y)\,dy\\
 &=\int_0^1 \sum_{m=-\infty}^\infty(\delta(x-y+2m)-\delta(x+y+2m))u_0(y)\,dy
 =\int_0^1 \delta(x-y)u_0(y)\,dy
 =u_0(x).\\
\end{aligned}
\end{equation*}
The other result follows directly from the definitions of $u_2$, $u_3$ and $u_4$.
\end{proof}

\begin{lemma}\label{lem:boundary_u14}
 $\;u_j(0,t)=u_j(1,t)=0$, for $\,j=1,4$ and $t\in (0,\infty)$.
\end{lemma}
\begin{proof}
Since $\theta_{(\mu)}(x,t)$ is even on $x$ which is stated 
 in Lemma~\ref{lem:uniform_theta}, then 
 \begin{equation*}
  u_1(0,t)=\int_0^1(\theta_{(\mu)}(0-y,t)-\theta_{(\mu)}(0+y,t))
  u_0(y)\,dy=0.
 \end{equation*}
 We also have 
 \begin{equation*}
  \begin{aligned}
   u_1(1,t)&=\int_0^1(\theta_{(\mu)}(1-y,t)-\theta_{(\mu)}(1+y,t))
  u_0(y)\,dy\\
  &=\int_0^1(\theta_{(\mu)}(y-1,t)-\theta_{(\mu)}(1+y,t))
  u_0(y)\,dy\\
  &=\int_0^1\!\Big[\!\sum_{m=-\infty}^\infty G_{(\mu)}(y-1+2m,t)-
  \!\!\sum_{m=-\infty}^{\infty}G_{(\mu)}(y+1+2m,t)\Big]u_0(y)\,dy\\
  &=\int_0^1\!\Big[\!\sum_{q=-\infty}^\infty G_{(\mu)}(y+1+2q,t)-
  \!\!\sum_{m=-\infty}^{\infty}G_{(\mu)}(y+1+2m,t)\Big]u_0(y)\,dy=0,
  \end{aligned}
 \end{equation*}
where $q=m-1$.
 
Following from the above proof, we obtain the conclusion for $u_4$.
\end{proof}

\begin{lemma}\label{lem:boundary_u23}
   $u_2(0,t)=g_0(t)$, $u_2(1,t)=0$,
   $u_3(0,t)=0$, $u_3(1,t)=g_1(t)$,
for $t\in (0,\infty)$.
\end{lemma}

\begin{proof}
 The proof follows from Lemma \ref{lem:boundary_overline_theta} directly. 
 \end{proof}

 \par Now we can state  
\begin{theorem}[Representation theorem]\label{thm:representation}
 There exists a unique solution $u(x,t)$ of Equations~\eqref{eqn:one_dim_model},
which has the representation 
$\;{\displaystyle u(x,t)=\sum_{j=1}^4 u_j}$.
\end{theorem}
\begin{proof}
The existence follows from Lemmas~\ref{lem:u1u2u3u4}, \ref{lem:initial_u}, 
\ref{lem:boundary_u14} and \ref{lem:boundary_u23}; while the uniqueness 
 is ensured by Corollary \ref{cor:existence_uniqueness}.
\end{proof}

\section{Determining the distributed coefficient $\mu(\alpha)$}
\label{inverse_problem}

In this section we state and prove two uniqueness theorems for the recovery
of the distributed  derivative $\mu$.
We show that by measuring the solution along a time trace from a fixed location
$x_0$ one can use this data to uniquely recover $\mu(\alpha)$.
This time trace can be one where the sampling point is located within the
interior of $\Omega=(0,1)$ and we measure $u(x_0,t)$, or we measure
the flux at $x^\star$; $u_x(x^\star,t)$ where $0<x^\star\leq 1$.
This latter case therefore includes measuring the flux on the right-hand
boundary $x=1$.

First we give the definition of the admissible set $\Psi$ 
according to Assumption~\ref{mu_assumption}.

\begin{definition}\label{def:Psi}
 Define the set $\Psi$ by 
 \begin{equation*}
  \Psi:=\{\mu\in C^1[0,1]:\ \mu\ge 0,  
  \ \mu(1)\ne 0,\ \mu(\alpha)\ge C_{\Psi}>0 \ \text{on}\ (\beta_0, \beta_1)\},
 \end{equation*}
where the constant $C_{\Psi}>0$ and the interval 
$(\beta_{0}, \beta_{1})\subset(0,1)$ 
only depend on $\Psi.$
\end{definition}

We introduce the functions 
$F(y;x_0)$ and $F_f(y;x^\star)$ in the next two lemmas. 
\begin{lemma}\label{lem:F}
Define the function $F(y;x_0)\in C^1((0,\infty),\mathbb{R})$ as 
$$ F(y;x_0)=\frac{e^{(x_0-2)y}-e^{-x_0 y}}{2(1-e^{-2y})},$$
where $x_0\in (0,1)$ is a constant. 
Then the function $F(y;x_0)$ is strictly increasing on the interval 
$(\frac{\ln (2-x_0)-\ln x_0}{2(1-x_0)}, \infty)\subset (0,\infty).$
\end{lemma}

\begin{proof}
Since $x_0\in (0,1)$, 
$e^{(x_0-2)y}-e^{-x_0 y}<0$ and $2(1-e^{-2y})>0$ on  $(0,\infty)$. 
A direct calculation now yields
$$\frac{d}{dy}(e^{(x_0-2)y}-e^{-x_0 y})=(x_0-2)e^{(x_0-2)y}+x_0e^{-x_0 y}>0$$
for $y\in (\frac{\ln (2-x_0)-\ln x_0}{2(1-x_0)}, \infty)$. 
Then we have $e^{(x_0-2)y}-e^{-x_0 y}<0$ and strictly increasing on 
 $(\frac{\ln (2-x_0)-\ln x_0}{2(1-x_0)}, \infty)$. 
The function $2(1-e^{-2y})$ is  obviously both positive and strictly 
 increasing on $(\frac{\ln (2-x_0)-\ln x_0}{2(1-x_0)}, \infty)$.
Hence the function $F(y;x_0)$ is also strictly increasing on
$(\frac{\ln (2-x_0)-\ln x_0}{2(1-x_0)}, \infty)$, which completes 
 the proof.
 \end{proof}

 \begin{lemma}\label{lem:F_f}
 For the inverse problem with flux data, 
 define the function $F_f(y;x^\star)\in C^1((0,\infty),\mathbb{R})$ as 
 $$F_f(y;x^\star)=\frac{y e^{(x^\star-2)y}+ye^{-x^\star y}}
  {2(1-e^{-2y})},$$
  where $x^\star\in (0,1]$ is a constant. 
  Then the function $F_f(y;x^\star)$ is strictly decreasing on the interval 
  $(1/x^\star,\infty)\subset (0,\infty).$
\end{lemma}

\begin{proof}
 \begin{equation*}
 \begin{split}
  \frac{\partial F_f}{\partial y}(y;x^\star )
  &=\frac{((x^\star -2)y+1)e^{(x^\star -2)y}+(1-x^\star y)e^{-x^\star y}}{2(1-e^{-2y})^2}\\
  &\quad +\frac{(-x^\star y-1)e^{(x^\star -4)y}
  +((x^\star -2)y-1)e^{(-x^\star -2)y}}{2(1-e^{-2y})^2},
 \end{split} 
 \end{equation*}
hence $\frac{\partial F_f}{\partial y}(y;x^\star )<0$ if
$y\in (1/x^\star,\infty)$ and the proof is complete. 
\end{proof}

For the important lemmas to follow, we need the Stone--Weierstrass and the
M\"untz--Sz\'asz Theorems.
See the appendix for statements and references for these results.

The next result shows that the set
$\{(n r)^x:\N+\}$ is complete in $L^2[0,1]$ for any positive integer $r$.
We give two proofs of this important lemma.

\begin{lemma}\label{lem:dense}
   For each $r\in \mathbb{N}^+,$ the vector space consisting with the 
   set of functions $\{(nr)^x:\ \N+\}$ is dense in the space 
   $L^2[0,1],$ i.e.
  $$\overline{span\{(nr)^x:\N+\}}=L^2[0,1]$$
  w.r.t $L^2$ norm. 
  In other words, the set $\{(nr)^x:\N+\}$ is complete in $L^2[0,1].$
 \end{lemma}

\begin{proof}
Clearly, $span\{(nr)^x:\N+\}$ satisfies all the conditions of the Stone--Weierstrass Theorem, so that the closure of $span\{(nr)^x:\N+\}$ w.r.t the 
continuous norm is either $C[0,1]$ or $\{f\in C[0,1]:f(x_0)=0, x_0\in[0,1]\}.$
  The two alternatives both yield that 
  $span\{(nr)^x:\ \N+\}$ is dense in $C[0,1]$ with respect to the $L^2$ norm, 
  which together with the fact $C[0,1]$ is dense in $L^2[0,1]$ gives 
  $span\{(nr)^x:\N+\}$ is dense in $L^2[0,1]$ and completes the proof.

As a second proof,
if for some $h\in C[0,1]$, $\int_0^1 (n r)^x h(x)\,dx = 0$ for all 
$n\in \mathbb{N}^+$ then 
$\int_0^1 e^{x\log(r n) } h(x)\,dx = 0$ and with the change of variables
$y = e^{x}$ this becomes 
$\int_1^e y^{\log(r n)}\tilde h(y)\,dy = 0$ for all $n\in \mathbb{N}^+$
where $\tilde h(y) = h(\log(y))/y$.
Since $\sum_{n=1}^\infty 1/\log(r n)$ diverges, the M\"untz-Sz\'asz theorem
shows that $\tilde h =0$ and hence $h(x) = 0$.
\end{proof}

We now have the main result of this paper.

\begin{theorem}[Uniqueness theorem for the inverse problem]\label{thm:uniqueness_mu}
In the  DDE~\eqref{eqn:one_dim_model}, set $u_0=g_1=f=0$ and let $g_0$
satisfy the following condition 
\begin{equation*}\label{eqn:condition_g_0_v1}
   (\L g_0)(z)\ne 0\ \text{for}\ z\in(0,\infty).
  \end{equation*}
Given $\mu_1$, $\mu_2\in \Psi$, denote the two weak solutions 
with respect to
$\mu_1$ and $\mu_2$ by $u(x,t;\mu_1)$ and $u(x,t;\mu_2)$ respectively.  
Then for any $x_0\in (0,1)$ and $x^\star\in(0,1]$, either
\begin{equation*}\label{interior_data}
u(x_0,t;\mu_1)=u(x_0,t;\mu_2)
\end{equation*}
or
\begin{equation*}\label{flux_data}
\frac{\partial u}{\partial x} 
(x^\star,t;\mu_1)=\frac{\partial u}{\partial x}(x^\star,t;\mu_2),\ t\in(0,\infty)
\end{equation*}
implies $\mu_1=\mu_2$ on $[0,1]$.
 \end{theorem}

\begin{proof}
For the first case of $u(x_0,t;\mu_1)=u(x_0,t;\mu_2),$ fix $x_0\in (0,1)$, 
Theorem~\ref{thm:representation} yields for $k=1,\,2$:
 \begin{equation*}
u(x_0,t;\mu_k)=-2\int_0^t \overline{\theta}_{(\mu_k)}(x_0,t-s)g_0(s) \,ds,
\qquad k=1,\; 2
 \end{equation*}
which implies 
$$
\int_0^t \overline{\theta}_{(\mu_1)}(x_0,t-s)g_0(s)\,ds
=\int_0^t \overline{\theta}_{(\mu_2)}(x_0,t-s)g_0(s)\,ds.
$$
Taking the Laplace transform in $t$ on both sides of the above equality gives
$$\Big(\L(\overline{\theta}_{(\mu_1)}(x_0,\cdot))\Big)(z)\cdot (\L g_0)(z)
=\Big(\L(\overline{\theta}_{(\mu_2)}(x_0,\cdot))\Big)(z)\cdot (\L g_0)(z).$$
Since $(\L g_0)(z)\ne 0\ \text{on}\ (0,\infty),$ 
so that 
$$\Big(\L(\overline{\theta}_{(\mu_1)}(x_0,\cdot))\Big)(z)
=\Big(\L(\overline{\theta}_{(\mu_2)}(x_0,\cdot))\Big)(z), 
\ \text{for}\ z\in (0,\infty).$$
This result and \eqref{eqn:L_theta} then give
$$
\frac{e^{(x_0-2)\Phi_1^{1/2}(z)}-e^{-x_0\Phi_1^{1/2}(z)}}{2(1-e^{-2\Phi_1^{1/2}(z)})}
=\frac{e^{(x_0-2)\Phi_2^{1/2}(z)}-e^{-x_0\Phi_2^{1/2}(z)}}
  {2(1-e^{-2\Phi_2^{1/2}(z)})},
  \ z\in(0,\infty),
$$
where 
$$\Phi_j (z)=\int_0^1 \mu_j(\alpha)z^\alpha {\rm d}\alpha,\quad j=1,2.$$
The definition of $\Psi$ and the fact $z\in (0,\infty)$ yield 
$\Phi_j^{1/2}(z)\in (0,\infty)$ and hence we can rewrite the above 
equality as 
\begin{equation}\label{eqn:equality_F}
F(\Phi_1^{1/2}(z);x_0)=F(\Phi_2^{1/2}(z);x_0),\ z\in (0,\infty),
\end{equation}
where the function $F$ comes from Lemma~\ref{lem:F}.

Since $x_0\in (0,1)$, it is obvious that
$\displaystyle{\frac{\ln (2-x_0)-\ln x_0}{2(1-x_0)}>0}$. 
Then we can pick a large $N^*\in\mathbb{N}^+$ such that 
\begin{equation*}
 \int_{\beta_0}^{\beta_1} C_\Psi\cdot (N^*)^\alpha {\rm d}\alpha 
 >\left(\frac{\ln (2-x_0)-\ln x_0}{2(1-x_0)}\right)^2,
\end{equation*}
which together with the definition of $\Psi$ gives that  
for each $z\in (0,\infty)$ with $z\ge N^*,$ 
$\Phi_j(z)\in (0,\infty)$ and  
$$\Phi_j^{1/2}(z)>\frac{\ln (2-x_0)-\ln x_0}{2(1-x_0)},\quad j=1,2.$$
This result means that 
\begin{equation}\label{eqn:inequality_Phi}
\Phi_j^{1/2}(nN^*)>\frac{\ln (2-x_0)-\ln x_0}{2(1-x_0)},\ j=1,2,\ \N+.
\end{equation}
Lemma~\ref{lem:F} shows that $F(\cdot;x_0)$ is strictly increasing on the 
interval $\bigl(\frac{\ln (2-x_0)-\ln x_0}{2(1-x_0)}, \infty\bigr)$, 
which together with \eqref{eqn:equality_F} and \eqref{eqn:inequality_Phi} 
yields 
$$\Phi_1^{1/2}(nN^*)=\Phi_2^{1/2}(nN^*),\ \N+,$$
that is
$\Phi_1(nN^*)=\Phi_2(nN^*),\ \N+$,
sequentially, we have 
\begin{equation*}
 \int_0^1 (\mu_1(\alpha)-\mu_2(\alpha)) (nN^*)^\alpha {\rm d}\alpha=0,
 \ \N+.
\end{equation*}
We can rewrite the above result as 
$\,\langle \mu_1(\alpha)-\mu_2(\alpha),(nN^*)^\alpha\rangle=0$ for $\N+$.
From the completeness of $\{(nN^*)^\alpha:\N+\}$ in $L^2[0,1]$ which is 
ensured by Lemma~\ref{lem:dense}, we have $\mu_1-\mu_2=0$ in $L^2[0,1]$, 
that is,
$\,\|\mu_1-\mu_2\|_{L^2[0,1]}=0$,
which together with the continuity of $\mu_1$ and $\mu_2$ shows that 
$\mu_1=\mu_2$ on $[0,1].$ 

For the case of $\frac{\partial u}{\partial x} 
(x^\star,t;\mu_1)=\frac{\partial u}{\partial x}(x^\star,t;\mu_2),$ 
following \eqref{eqn:L_theta} we have 
  \begin{equation*}\label{eqn:L_theta_x}
  \begin{split}
   &\quad\ \L\left(\frac{\partial \overline{\theta}_{(\mu)}}{\partial x}(x,t)\right)
   =\L \left[\kappa * \left(\frac{\partial ^3}{\partial t\partial x^2}
   \sum_{m=-\infty}^{\infty} G_{(\mu)}(x,t)\right)\right]\\
   &=\L \left[\kappa *\L^{-1}\left(\sum_{m=-1}^{-\infty} 
   \frac{\Phi^{3/2}(z)}{2}  e^{\Phi^{1/2}(z)(x+2m)}{\rm d}z
  +\sum_{m=0}^{\infty}\frac{\Phi^{3/2}(z)}{2}
  e^{-\Phi^{1/2}(z)(x+2m)}\right)\right]\\
  &=\frac{1}{\Phi(z)}\left(\sum_{m=-1}^{-\infty} \frac{\Phi^{3/2}(z)}{2}e^{\Phi^{1/2}(z)(x+2m)}
  +\sum_{m=0}^{\infty} \frac{\Phi^{3/2}(z)}{2}e^{-\Phi^{1/2}(z)(x+2m)}\right)\\
  &=\frac{\Phi^{1/2}(z)e^{(x-2)\Phi^{1/2}(z)}+\Phi^{1/2}(z)e^{-x\Phi^{1/2}(z)}}
  {2(1-e^{-2\Phi^{1/2}(z)})}.
  \end{split}
 \end{equation*}
Following the proof for the case $u(x_0,t;\mu_1)=u(x_0,t;\mu_2)$, 
we can deduce $\mu_1=\mu_2$ from the above result and Lemmas \ref{lem:F_f}
and \ref{lem:dense}. 
\end{proof}

\begin{remark}
In this paper we have considered only the uniqueness question for the function
$\mu(\alpha)$.
Certainly, one would like to know under what conditions this function
can be effectively recovered from the given data.
Clearly this is an important question, but we caution there are many
difficulties, especially with a mathematical analysis of the stability issue
of $\mu$ in terms of the overposed data either $u(x_0,t)$ or $\frac{\partial u}{\partial x}(x^\star,t)$.
One can certainly employ the representation result of
section~\ref{sect:representation} to obtain a nonlinear integral equation for
$\mu$ but the analysis of this is unclear.
An alternative approach would be restrict the function $\mu$ as in
Lemma~\ref{lem:kappa} to ensure that  $\kappa$ is completely monotone and hence
use Bernstein's theorem to obtain an integral representation for this function.
We hope to address some of these questions in subsequent work.
\end{remark}


\section*{Appendix}
The uniqueness proof in section \ref{inverse_problem} requires results on the density of
a certain subset of functions and we give two ways to look at this through
different formulations; namely the Stone-Weierstrass and M\"untz-Sz\'asz
theorems. 
We give the statements of these results below.

The Stone-Weierstrass theorem is a generalization of Weierstrass' result of
1885 that the polynomials are dense in $C[0,1]$ and was proved by
Stone some 50 years later, \cite{Stone:1948}.
If $X$ is a compact Hausdorff space and $C(X)$ those real-valued continuous
functions on $X$, with the topology of uniform convergence, then the question
is when is a subalgebra $A(X)$ dense?
A crucial notion is that of separation of points;
a set $A$ of functions defined on $X$ is said to {\it separate points\/} if,
for every $x,y\in X$, $x\not=y$,
there exists a function $f\in A$ such that $f(x) \not= f(y)$. 
Then we have

\begin{theorem}(Stone--Weierstrass).
Suppose $X$ is a compact Hausdorff space and $A$ is a subalgebra of
$C(X)$ which contains a non-zero constant function.
Then $A$ is dense in $C(X)$ if and only if it separates points.
\end{theorem}

The proof can be found in standard references, for example,
\cite[Theorem 4.45]{Folland:2013}.

The M\"untz-Sz\'asz theorem, (1914-1916) is also a generalization of
the Weierstrass approximation theorem; it gives a condition under which
one can ``thin out'' the polynomials and still maintain a dense set.

\begin{theorem}(M\"untz--Sz\'asz)
Let $\Lambda := \{\lambda_j\}_1^\infty$ be a sequence of real positive numbers.
Then the span of $\{1,x^{\lambda_1},x^{\lambda_2},\ldots\,\}$
is dense in $C[0,1]$ if and only if
$\sum_1^\infty\frac{1}{\lambda_j} = \infty$.
\end{theorem}

This result can be generalized to the $L^p[0,1]$ spaces for $1\leq p\leq \infty$,
see \cite{BorweinErdelyi:1996}.


\section*{Acknowledgment}
The authors were partially supported by NSF Grant DMS-1620138.

\bibliographystyle{abbrv}
\bibliography{frac}

\end{document}